\documentclass[12pt,a4paper,leqno]{amsart}
\usepackage[utf8]{inputenc}
\usepackage{amssymb,color,tikz, cancel}
\usepackage{amsmath,amscd,amsfonts,amsthm,verbatim}
\usepackage{tikz-cd}
\usepackage{enumitem}
\usepackage[all]{xy}
\usepackage[toc]{appendix}
\usepackage{xcolor}

\newcommand{\lra}{{\longrightarrow}}
\newcommand{\cK}{{\mathcal K}}
\newcommand{\cC}{{\mathcal C}}

\newcommand{\cE}{{\mathcal E}}
\newcommand{\cS}{{\mathcal S}}
\newcommand{\cF}{{\mathcal F}}
\newcommand{\cG}{{\mathcal G}}

\newcommand{\cO}{{\mathcal O}}

\newcommand {\PP}{\mathbb{P}}

\newcommand{\cH}{{\mathcal H}}

\DeclareMathOperator{\Spl}{Spl}

\DeclareMathOperator{\rk}{rk}

\DeclareMathOperator{\Ext}{Ext}

\DeclareMathOperator{\Hom}{Hom}

\DeclareMathOperator{\coker}{coker}

\newtheorem{theorem}{Theorem}[section]
\newtheorem{lemma}[theorem]{Lemma}
\newtheorem{question}[theorem]{Question}
\newtheorem{problem}[theorem]{Problem}
\newtheorem{proposition}[theorem]{Proposition}
\newtheorem{corollary}[theorem]{Corollary}
 \theoremstyle{definition}
\newtheorem{definition}[theorem]{Definition}
\newtheorem{example}[theorem]{Example}
\newtheorem{remark}[theorem]{Remark}

\makeatletter
\ifnum\@ptsize=0  \fi \ifnum\@ptsize=2  \fi 

\title[Matrices of constant rank]{Characterization of linear spaces of matrices of  constant rank from  syzygy bundles}
\author[S. Marchesi]{Simone Marchesi}
\address{Facultat de
  Matem\`atiques i Inform\`atica, Universitat de Barcelona, Gran Via des les
  Corts Catalanes 585, 08007 Barcelona, Spain}
\email{marchesi@ub.edu}
 \author[R.\ M.\ Mir\'o-Roig]{Rosa M.\ Mir\'o-Roig}
  \address{Facultat de
  Matem\`atiques i Inform\`atica, Universitat de Barcelona, Gran Via des les
  Corts Catalanes 585, 08007 Barcelona, Spain} \email{miro@ub.edu, ORCID 0000-0003-1375-6547}

\date{\today}


\begin{document}

\begin{abstract}

    In this work, we characterize matrices of linear forms and constant rank, demonstrating that, under some natural assumptions, they are always associated with a syzygy bundle that fits into a (partially linear) resolution. Furthermore, this construction 
    allows us to list all indecomposable matrices of constant rank up to 7, as well as describing the moduli spaces of simple vector bundles naturally defined by families of constant rank matrices.
\end{abstract}

\maketitle

\section{Introduction}
The construction of linear spaces of matrices with constant rank is considered a classical problem; while its formulation is elementary, we are still far from achieving a complete understanding of it. Moreover, its interest is of great relevance, due to the several interpretations and ramifications into which this problem can be translated into.
Its origins can be traced back to the work of Kronecker and Weierstrass (see \cite{W} and \cite{W2}). As observed in the seminal works of Sylvester \cite{S} and Eisenbud and Harris  \cite{EH}, the search for families of constant matrices can be reformulated as a vector bundle problem.  
This perspective naturally connects the study of constant rank matrices to two independently significant topics: vector bundles \textit{generated by global sections} and \textit{uniform bundles}, both of which have been extensively studied in the literature. Indeed, every linear space of constant rank matrices (equivalently, a matrix $M$ with linear entries and constant rank) corresponds to a uniform vector bundle $\cE_M$, such that both $\cE_M$ and $\cE_M^\lor(1)$ are generated by global sections. Conversely, given a vector bundle $\cE$ with these properties, we can reverse the construction to obtain a matrix $M_\cE$ with linear entries and constant rank.
Once this relation is established, it is important to recall that the classification of rank $r$ uniform vector bundles on $\PP^n$ is known for $r \leq n+1$. Within this range, for $n \geq 3$, they are sums of line bundles and/or a twist of either the tangent bundle or the cotangent bundle.  

This result plays a crucial role in the classification of matrices of linear forms with
constant rank $r\leq 6$ done in \cite{MMR}. That work strongly relies on the classification theorems of globally generated vector bundles with small first Chern class (see \cite{ACM} and \cite{ACM2}) together with properties derived from being uniform (see for example \cite{MarMR}) and the study of homogeneous bundles presented in \cite{LM}.

Furthermore, the impact of constant rank matrices extends even further. The results in \cite{LM} have recently been applied to the study of border rank bounds for tensors that are invariant under the general linear group (see \cite{WU}).

Being aware of how deep-seated the problem is, it is now essential to understand the methods used so far to construct the few known examples. We find out that the points of view and techniques involved are quite different from one method to another. In particular, Boralevi, Faenzi and Mezzetti (see \cite{BFM}) consider instanton bundles and derived categories to shed light on an isolated example given by Westwick in \cite{W2}. Their work represents a foundational step toward a more general construction, later 
 developed by Boralevi, Faenzi, and Lella in \cite{BFL}. This subsequent approach, rooted in commutative algebra, is based on the connections between the linear resolution of truncated modules. Very recently, Landsberg and Manivel (see \cite{LM}) use representation theory, while Manivel and Miró-Roig (see \cite{MMR}) rely, as mentioned before, on the classification of globally generated vector bundles with small first Chern class.
 
 Here, we  explore the construction proposed by Boralevi, Faenzi and Lella in \cite{BFL} using a vector bundle point of view of their result (see Theorem \ref{mainthm}). Roughly speaking, a matrix with the required properties appears when relating linear parts of the chain complexes, whose indexes are paired 
 by the chain complex map. In addition, we prove that, under some weak assumption, any constant rank matrix $M$ with linear entries can be constructed using this technique (see Theorem \ref{THM}).
 
 However, the primary objective of this work is to emphasize the importance of the particular case in which the map between two considered resolutions is defined by a \textit{contraction of global sections}. The contractions we will need must be of \textit{maximal rank}. This observation directly connects our problem to another significant open question in commutative algebra. Indeed, determining the rank of a given contraction is equivalent to the long-standing open problem of computing the Hilbert function of the exterior algebra modulo a principal ideal generated by a generic form of degree $q$ (for more information, the reader can look at \cite{LN} and \cite{MS}). 
 
 As a first\ direct application of our \textit{construction by contractions}, we provide a complete list of the indecomposable matrices of constant rank 7 which is a novelty (see Theorem \ref{thm-rank7}).
 The arguments involved rely on the classification of globally generated vector bundles with first Chern class $c_1\le 3$ combined with relating complexes which arise as the resolution of the most basic examples of simple regular bundles: the Koszul complex, that represents the resolution of $\cO_{\PP^n}(1)$, and, more in general, the Eagon-Northcott complexes that are obtained as the resolution of $\cO_{\PP^n}(d)$, with $d\geq 2$. Our argument can also be used to provide a more direct proof of the classification of matrices of constant rank $r \le 6$ given in \cite{MMR}. Recall moreover that the first syzygy  bundles obtained by contracting global sections of $\cO_{\PP^n}(d)$ are known as \textit{Steiner bundles} when $d = 1$, and as \textit{Drézet bundles} when $d \geq 2$, both of which are of independent interest.

  As mentioned before, the classical complexes considered arise as the resolution of the most basic case of vector bundles which are \textit{regular} and \textit{simple}. In the last part of this work, we focus on families of vector bundles that satisfy, respectively and/or separately, these properties.
  
  First, we prove a general result that ensures the construction of constant rank matrices with linear entries starting from the syzygies of regular bundles (see Proposition \ref{prop-regmatrix}). Again, the central idea is to find an adequate contraction of the global sections of the regular bundle $\cF$ whose kernel still generates $\cF$.

Regarding simplicity, we provide a result which does not depend on a chosen contraction if a resolution is ensured. Indeed, it is known that if we start with a simple vector bundle (satisfying some open conditions in cohomology), the syzygy bundles obtained in the splitting of its resolution will also be simple. The importance of this observation is twofold. First of all, the simplicity of the syzygy bundle $\cE_M$ implies the indecomposability of the associate matrix $M$  of linear forms and constant rank. Furthermore, it allows to describe the family of  the vector bundles $\cE_M$  inside the moduli space of simple vector bundles with fixed rank and fixed Chern classes. We will show that they are  parametrized by an open dense subset of an irreducible component (see Theorem \ref{moduli}).

\vskip 4mm

The paper is organized as follows: in Section \ref{sec-notation} we set the notation and recall the notions which will be used throughout this work, with special emphasis on contractions. Section \ref{sec-contruction} contains the construction of matrices of linear forms and constant rank through morphisms of acyclic complexes and its reverse construction (see Theorem \ref{mainthm} and Theorem \ref{THM}, respectively). In Section \ref{sec-applications}, we focus on the constructions defined by a given contraction of global sections, which allow us to provide the complete list of the indecomposable matrices of constant rank 7 (see Theorem \ref{thm-rank7}) and that we also apply to regular bundles. In Section \ref{sec-moduli}, we parametrize the family of matrices with linear entries and constant rank which arise from the resolution of simple vector bundles, and their moduli space (see Theorem \ref{moduli}). Section \ref{sec-openpb} contains questions and  problems stemming from this work. Finally, in the Appendix, we will show how to construct, using Macaulay 2, two specific examples: one coming from the contraction associated to a Drézet vector bundles, and the other one obtained starting from the Horrocks-Mumfors bundle.

\medskip \noindent {\it Acknowledgements}.   Both authors have been partially supported by the grant PID2020-113674GB-I00.

\section{Notation and General set up}\label{sec-notation}

Throughout this paper we work over an algebraically closed field $k$ of characteristic 0. Given a $k$-vector space $V$ of dimension $n+1$, we denote by $V^*=\Hom(V,k)$ its dual, by $S(V)$ the symmetric algebra and by $\bigwedge V$ the exterior algebra. We fix $\{v_i\}$ a basis of $V$ and we denote by $\{v_i^*\}$ its associate dual basis.
We define  $R=S(V^*)$ and we identify it with a polynomial ring, i.e., $R=k[x_0,\cdots ,x_n]$. For any graded $R$-module $N$, we define $N_{\ge p}:=\oplus _{i\ge p}N_{i}$ its truncation at degree $p$. 
We define $\PP^n=\PP (V)$ as the set of lines, and therefore we have $H^0(\PP^n, \cO_{\PP^n}(1))=V^*$. Finally, considering the tangent bundle $T=T_{\PP^n}$, we set $\Omega ^j=\bigwedge ^j T^\lor$.

\subsection{Constant rank matrices and vector bundles}
We now describe the mutual relation between linear space of constant ranks and the associated vector bundle.
Given two $k$-vector spaces $A$ and $B$ of dimension $a$ and $b$, respectively, we consider a $(n+1)$-dimensional linear subspace $M$ of $A^*\otimes B\cong \Hom(A,B)$ and we say that $M$ has \textit{constant rank} $r$ if every non-zero element of $M$ has rank $r$. Fixing a basis for $A$ and $B$, we can write $M$ as a $(b\times a$)-matrix whose entries are linear forms in $n+1$ variables, that we keep denoting by the same letter $M$.  This matrix defines a natural map (which we still denote by $M$) that induces the following diagram
\begin{equation}\label{maindiagram}
\xymatrix{
0 \ar[r] & \cK_M \ar[r] & A\otimes \cO_{\PP^n}\ar[rr]^M\ar[rd] &
& B\otimes \cO_{\PP^n}(1) \ar[r] & \cC_M \ar[r] & 0  \\
& & & \cE_M \ar[ru]& && \\
}\end{equation}
\noindent where $\cK_M$, $\cC_M$ and $\cE_M$ denote, respectively, the kernel, cokernel and image of $M$. Being $M$ of constant rank $r$ is equivalent to say that $\cK_M$, $\cC_M$ and $\cE_M$ are vector bundles on $\PP^n$ of ranks $a-r$, $b-r$ and $r$ respectively.

We will say that $M$ is \textit{indecomposable} if the associated rank $r$ vector bundle $\cE _M$ on $\PP^n$ is indecomposable. To ensure this property, we will often deal with \textit{simple} vector bundle, i.e., whose only endomorphims are the
homotheties. Recall that any simple bundle is indecomposable but not vice versa. In Section \ref{sec-moduli}, we will discuss open cohomological conditions that ensure the simplicity, and thus the  indecomposability, of the vector bundles involved.

Directly from Diagram (\ref{maindiagram}), we have that $\cE_M$ and $\cE_M^\lor(1)$ are generated by global sections, which implies  that $\cE_M$ has the same splitting type for every line in $\PP^n $, i.e.,
$$
(\cE_M)_{|L}\simeq \cO_L(1)^{c_1(E)}\oplus \cO_L^{(r-c_1(E))} \mbox{ for every } L\subset \PP^n,
$$
in which case the vector bundle is called {\em uniform.}

Conversely, given a rank $r$ vector bundle $\cE$, such that $\cE_M$ and $\cE_M^\lor(1)$ are both globally generated, we can revert the previous construction through the constant rank $r$ natural morphism defined as
$$
\psi_\cE: H^0(\PP^n,\cE)\otimes\cO_{\PP^n}\longrightarrow H^0(\PP^n,\cE^\vee(1))^\vee\otimes\cO_{\PP^n}(1).$$



\begin{definition} \label{regular}
 A coherent sheaf $\cF$ on $\PP^n$ is called \textit{(Castelnuovo-Mumford) regular} if
$$
H^i(\cF(-i))=0 \:\:\mbox{ for any } \:\: i>0.
$$
We will say that $\cF$ is \textit{$m$-regular} if $\cF(m)$ is a regular sheaf and $\cF(m-1)$ is not. 
\end{definition}
As an  application of the Beilinson spectral sequence, we have the following resolution for any regular vector bundle $\cF$ on $\PP^n$ (see for example \cite[II.3.1]{OSS}).:
\begin{equation}\label{res-regsheaf}
\begin{array}{ccc}
0 \longrightarrow H^0(\cF(-1))\otimes \cO_{\PP^n}(-n) \longrightarrow H^0(\cF\otimes \Omega^{n-1}(n-1))\otimes \cO_{\PP^n}(-n+1) \longrightarrow \cdots\vspace{2mm}\\
\cdots \longrightarrow H^0(\cF\otimes \Omega^{p}(p))\otimes \cO_{\PP^n}(-p)\longrightarrow H^0(\cF\otimes \Omega^{p-1}(p-1))\otimes \cO_{\PP^n}(-p+1)\longrightarrow \cdots\vspace{2mm}\\
\cdots  \longrightarrow H^0(\cF\otimes \Omega^{1}(1))\otimes \cO_{\PP^n}(-1) \longrightarrow H^0(\cF)\otimes \cO_{\PP^n} \longrightarrow \cF \longrightarrow 0.
\end{array}
\end{equation}

\begin{example}\label{classic}  \rm
   Let $V$ be a $k$-vector space of dimension $n+1$ and $\PP^n=\PP(V)$. A classical example of morphism of constant rank appears in the Koszul complex
   $$
   \cdots \rightarrow \wedge ^{i+1}V^*\otimes \cO_{\PP^n} \xrightarrow{\,\,\,M\,\,}  \wedge ^{i}V^*\otimes \cO_{\PP^n}(1) \rightarrow  \cdots
   $$
   where $M$ is an indecomposable $\left(\binom{n+1}{i}\times \binom{n+1}{i+1}\right)$ matrix of constant rank ${n\choose i}$ and $\cK_M$, $\cE_M$ and $\cC_M$ are uniform bundles of ranks ${n\choose i+1}$,  ${n\choose i}$ and ${n\choose i-1}$, respectively.  Indeed, we have $\cK_M=\Omega ^{i+1}(i+1)$, $\cE_M=\Omega ^i (i+1)$ and $\cC_M=\Omega ^{i-1}(i+1)$.
\end{example}

\begin{remark}
    Throughout the paper we will often use, without mentioning it again, that if $\cC _M$ is simple and  $H^1(\cC_M(-1))=H^1(\cC_M^\lor(1))=H^2(\cC_M^\lor(1))=0$, then $\cE_M$ is also simple and $M$ indecomposable (for concrete details see \cite{FM}). 
\end{remark}

\subsection{Contractions}
We conclude this section by recalling the definition of contraction in an exterior algebra, which will play a central role in the explicit construction technique we aim to highlight in this work.

\begin{definition}\label{def-contraction}
For any integer $q \leq n+1$ and any $q$-form  $v_{j_1}^* \wedge \cdots \wedge v_{j_q}^* \in \bigwedge^q V^*$, we define two different \textit{contractions}:\\
    the \textit{left contraction}
    $$
    \begin{array}{rccl}
         \left((v_{j_1}^* \wedge \cdots \wedge v_{j_q}^*)\wedge - \right) : & \bigwedge^p V & \longrightarrow & \bigwedge^{p-q} V \vspace{1mm}\\
         & u & \mapsto &
         \left\{
         \begin{array}{cl}
             (-1)^{\sigma_u}w & \mbox{ if } u = (-1)^{\sigma_u}(v_{j_1} \wedge \cdots \wedge v_{j_q})\wedge w  \\
            0 & \mbox{ otherwise.}
         \end{array}
         \right.
    \end{array}
    $$
    and \textit{right contraction}
    $$
    \begin{array}{rccl}
         \left(- \wedge (v_{j_1}^* \wedge \cdots \wedge v_{j_q}^*)\right) : & \bigwedge^p V & \longrightarrow & \bigwedge^{p-q} V \vspace{1mm}\\
         & u & \mapsto &
         \left\{
         \begin{array}{cl}
             (-1)^{\sigma_u}w & \mbox{ if } u = (-1)^{\sigma_u} w \wedge (v_{j_1} \wedge \cdots \wedge v_{j_q})  \\
            0 & \mbox{ otherwise.}
         \end{array}
         \right.
    \end{array}
    $$
  By linearity, we extend both definitions to any  $q$-form $\omega\in \bigwedge^q V^*$. We will say that a $q$-form $\omega\in \bigwedge^q V^*$ defines a \textit{contraction of maximal rank} if  the associated linear map has maximal  rank for every $p=q,\ldots,n+1.$
\end{definition}
     
    Fixed $s\cdot t$ elements in $\bigwedge^q V^*$ we can define analogously the concept of contraction of maximal rank, relative to the linear map
    $$
    \left(\bigwedge^p V\right)^{ t}  \longrightarrow  \left(\bigwedge^{p-q} V \right)^{ s}
    $$
    induced by the fixed elements.
    Notice that the rank of a contraction on the left is maximal if and only if it is so for the paired contraction on the right.

The following example  provides us infinite families of maximal rank contractions.

    \begin{example}\label{ex-maxrank}
        If $n+1$ and $q$ are even integers, then the contraction defined by the generic element of $\bigwedge^q V^*$ is of maximal rank.   Indeed, we denote the basis of $V$ by $v_1,...,v_{2m}$.
 Since the $s$-th power of the non degenerated 2-form
        $$
        \omega=v_1^*\wedge v_{m+1}^*+v_2^*\wedge v_{m+2}^*+...+v_m^*\wedge
v_{2m}^*
        $$
        defines a contraction of maximal rank, the same holds for the generic $q$-form $\omega\in \bigwedge^q V^*$, as required.
 \end{example}
    
All the examples given in this work have maximal rank. Nevertheless, this is not always the case, as the following example, provided by Ottaviani and Weyman, shows.
\begin{example}\label{ex-notmaxrank} (See \cite[Remark 3.5]{GS})
Let $V$ be a 9-dimensional vector space. Then, the contraction $\bigwedge^6 V \longrightarrow \bigwedge^3 V$, defined by the generic $3$-form  $\omega$ in $\bigwedge^3 V^*$, is not of maximal rank. Indeed, it is represented by a $(84\times 84)$-matrix of rank $80$.
\end{example}

\section{Construction of matrices of constant rank}\label{sec-contruction}

In this section, we present the construction of matrices with linear entries and constant rank introduced in \cite{BFL} from a vector bundle perspective, which is particularly suitable for our purposes. We then examine whether this construction is reversible, that is, whether a constant-rank matrix can be obtained using this technique.

A \textit{bounded complex} $P_{\bullet}$ of coherent sheaves on $\PP^n$ is a sequence of coherent sheaves
  $$P_{\bullet}: \cdots \rightarrow P_{i-i}\xrightarrow{\,\,\,d_P^{i-1},\,\,}  P_i \xrightarrow{\,\,\,d_P^i\,\,\,}  P_{i+1}\rightarrow \cdots $$
such that, for any integer $i$, we have $d_P^id_P^{i-1}=0$ and $P_i=0$ for $i\gg 0$ and $i\ll 0$. We  set $p_i:=\rk(P_i)$.   The complex $P_{\bullet}$ is \textit{acyclic} if it is exact everywhere, i.e., $ker(d_P^i)=Im(d_P^{i-1})$ for all $i$ and 
\textit{linear} if   all differential maps
$d_P^i:P_i\rightarrow P_{i+1}$ can be represented by a $p_{i+1}\times p_{i}$ matrix with linear entries. We  say that $P_{\bullet}$ is \textit{$(r,s)$-linear} if $P_i=\cO_{\PP^n}(i)^{p_i}$ for any $r\le i \le s$, which implies that all differential maps
$d_P^i$, with $r\le i<s$, can be represented by a matrix with linear entries.
We denote by   $P[j]_{\bullet}$ the complex obtained by shifting the degree $j$ places to the left, i.e., $P[j]_i=P_{j+i}$.
We set $P_i(j):=P_i\otimes \cO_{\PP^n}(j)$ and 
$d^i_P(j):=d^i_P\otimes 1_{\cO_{\PP^n}(j)}$.

Without loss of generality, we can suppose our complexes to be non zero in degrees between $-n-1$ and $0$.

\begin{theorem}\label{mainthm}
     Let
  $P_{\bullet}$ and $Q_{\bullet}$ be two  acyclic  bounded complexes of vector bundles on $\PP^n$ and  let  $ \varphi_{\bullet }: P _{\bullet}
  \longrightarrow Q_{\bullet}(-j)[j] $ be a morphism 
  of complexes. We assume:
   \begin{itemize}
       \item[(i)]  $P_{\bullet}$ is $(r_P,s_P)$-linear and $Q_{\bullet}$ is $(r_Q,s_Q)$-linear.
       \item[(ii)] There exists an integer $i$, $r_P\le i <s_P$, such that
        \begin{itemize}
            \item $r_Q\le i+j<s_Q$;\vspace{1mm}
            \item the morphisms $$\varphi_i:P_i\longrightarrow Q_{i+j}(-j)$$ and $$\varphi_{i+1}:P_{i+1}\longrightarrow Q_{i+j+1}(-j)$$ are surjective.
       \end{itemize}
   \end{itemize}

Then, there exists a $ b_{i+2} \times (p_{i+1}-q_{i+j+1})$ matrix $M$, with $b_{i+2} := \dim \ker \varphi_{i+2}$, linear entries and constant rank $\sum _{k=0}^{-i-2}(-1)^{k} p_{i+2+k} - \sum _{k=0}^{-i-j-2}(-1)^{k} q_{i+j+2+k}$.
\end{theorem}

\begin{proof}

For any integer $i$, with $-n+1\le i \le -1$,  we set $E_i:=Ker(d_P^i)=Im(d_P^{i-1})$ and $G_i=: Ker(d_Q^i)=Im(d_Q^{i-1})$, respectively.
Therefore, we can split the given complexes in exact sequences
$$
0 \longrightarrow E_{i}  \longrightarrow  P_{i}  \longrightarrow  E_{i+1}  \longrightarrow 0 
$$
$$
0 \longrightarrow G_{i}  \longrightarrow  Q_{i}  \longrightarrow  G_{i+1}  \longrightarrow 0 
$$
   which allow us to compute 
   $$e_i:=rk(E_i)=\sum _{k=0}^{-i}(-1)^{k}rk(P_{i+k})=\sum _{k=0}^{-i}(-1)^{k} p_{i+k}$$
   and  
   $$g_i:=rk(G_i)=\sum _{k=0}^{-i}(-1)^{k} q_{i+k}.
   $$ 
   The surjective morphisms $\varphi_i$ and $\varphi_{i+1}$ induce surjective morphisms
   $$ \widetilde{\varphi _i}:E_{i}\longrightarrow  G_{i+j}(-j)  \quad \text{ and } \quad 
   \widetilde{\varphi_{i+1}}:E_{i+1}\longrightarrow  G_{i+j+1}(-j).$$ 
   Setting $F_{s}:=ker(\widetilde{\varphi_s}:E_s\longrightarrow G_{s+j}(-j))$, consider the commutative diagram
$$
\xymatrix{
& 0 \ar[d] & 0 \ar[d] & 0 \ar[d] \\
0 \ar[r] & F_{i+1}(-i-1) \ar[d] \ar[r] & E_{i+1}(-i-1) \ar[d] \ar[r] & G_{i+1+j}(-i-j-1)\ar[d] \ar[r] & 0\\
0 \ar[r] & \cO_{\PP^n}^{p_{i+1}-q_{i+j+1}} \ar[r] \ar[d]  & P_{i+1}(-i-1) \ar[r] \ar[d] & Q_{i+1+j}(-i-j-1) \ar[r] \ar[d] & 0\\
0 \ar[r] & F_{i+2}(-i-1) \ar[r] \ar[d] & E_{i+2}(-i-1) \ar[d] \ar[r] & G_{i+j+2}(-i-j-1) \ar[r] \ar[d] & 0\\
& 0 & 0 & 0
}
$$

from which we can compute, for $i+1\le s\le i+2$, 
$$
f_s:=rk(F_s)=rk(E_s)-rk(G_{s+j})=
  \sum _{k=0}^{-s}(-1)^{k} p_{s+k} - \sum _{k=0}^{-s-j}(-1)^{k} q_{s+k+j},
  $$ and induce the following one

$$\xymatrix{
F_{i+1}(-i-1) \ar[rd] &&& &  F_{i+3}(-i-1) & \\
 & \cO_{\PP^n}^{p_{i+1}-q_{i+j+1}}\ar[rr]^M \ar[rd] &
&\cO_{\PP^n}(1)^{b_{i+2}}\ar[ru] & &\\
 & &  F_{i+2}(-i-1) \ar[ru]& && \\
 }$$
with $b_{i+2} = \dim \ker \varphi_{i+2}$. Thus the result follows.
\end{proof}

\begin{remark}
We do not know, in general, whether the matrix $M$ constructed in Theorem  \ref{mainthm} is indecomposable. Quite often the matrix $M$ will be associated to a syzygy bundle $F_{i+2}(-i-2):=\ker(\cO_{\PP^n}^{a_{i+2}} \longrightarrow   F_{i+3}(-i-2))$  of a simple bundle  $F_{i+3}(-i-2)$. Therefore, $F_{i+2}$ will also be simple and $M$ indecomposable.
\end{remark}

In the next section, we will introduce families of complexes that satisfy all the hypotheses of the above result and examine trickier constructions in detail. Meanwhile, we present a couple of examples to better illustrate the proposed construction.
\begin{example}\label{Firstexamples}
Consider the Koszul complex
  $$P_{\bullet}: 0\rightarrow P_{-n-1}\rightarrow P_{-n}\rightarrow \cdots \rightarrow P_{-2}\rightarrow P_{-1}\rightarrow P_0\rightarrow 0 $$ with 
  $P_i=\cO_{\PP^n}(i)^{{n+1\choose -i}}$.
  The hypotheses of Theorem \ref{mainthm} are satisfied for the following cases:\\

 \noindent   (i) Consider $Q_{\bullet} = 0$ and $j=0$. We recover the classical example of  an indecomposable $\binom{n+1}{i}\times \binom{n+1}{i+1}$ matrix with linear entries and  constant rank ${n\choose i}$ (see Example \ref{classic}).\\
   
  \noindent  (ii) Consider $P_{\bullet}=Q_{\bullet}$ and any pair of integers $i\leq -1$ and $j \geq 1$ that admit a maximal rank contraction (implying the relation $\binom{n+1}{-i}\ge \binom{n+1}{-i-j}$).
    In particular:
    
    \vskip 2mm
    
    - For $n=5$, $j=2$ and $i=-4 $, we recover the example  \cite[Example 6.2]{LM} and we construct a $14\times 14$ indecomposable matrix $M$ with entries  linear forms in 6 variables and constant rank 9 which gives rise to the following diagram:
$$\xymatrix{
0 \ar[r] &\cK_M \ar[r] &\cO_{\PP^5}^{ 14}\ar[rr]^M\ar[rd] &
& \cO_{\PP^5}(1)^{14} \ar[r] & \Omega ^1(3) \ar[r] & 0  \\
& & & \cE_M \ar[ru]& && \\
}$$
\noindent where $\cK_M:=ker(\Omega^3(3)\rightarrow \Omega^1(1))$ and  $\cE_M$ is a rank 9 uniform vector bundle on $\PP^5$.
\vskip 2mm
    - For $n=7$, $j=1$ and $i=-4$, we apply Theorem \ref{mainthm} and we build  a $28\times 20$ indecomposable matrix $M$ with entries  linear forms in 8 variables and constant rank 14 given by the following diagram:
$$\xymatrix{
0 \ar[r] & \cK_M\ar[r] &\cO_{\PP^7}^{ 28}\ar[rr]^M\ar[rd] &
& \cO_{\PP^7}(1)^{ 20} \ar[r] & \cC_M\ar[r] & 0  \\
& & & \cE_M \ar[ru]& && \\
}$$
\noindent where $\cK_M:=ker(\Omega^3(3)\rightarrow \Omega ^2(2))$ and  $\cC_M:=ker(\Omega^1(3)\rightarrow \cO_{\PP^7}(2))$.
In this case $\cE_M$ is a rank 14 uniform vector bundle on $\PP^7$.    
\end{example}

The last goal of this section is to prove that under some mild assumption any indecomposable matrix $M$ of constant rank $r$ which sits in a diagram  of type (\ref{maindiagram})
 can be constructed as in Theorem \ref{mainthm}.
 
Notice  that the rank $b-r$ cokernel vector bundle $\cC_M(-1)$ is generated by its global section. However, in general, $\cC_M(-1)$ is not 0-regular. To measure how far is $\cC_M (-1)$ from being 0-regular we introduce the following definition.
 
 \begin{definition}
     Let $\cE$ be a vector bundle on $\PP^n$ and fix an integer $p\ge 1$. We will say that $\cE$ has the \textit{property}   $L_p$ if it satisfies the following two conditions:
     \begin{itemize}
         \item[(i)] $\cE$ is globally generated; 
         \item[(ii)] the first $p$ steps of a minimal free $R$-resolution of $H^0_*(\cE)$ are linear, i.e., it has the following shape:
$$
\cdots \rightarrow R(-p+1)^{a_{p-1} }\rightarrow \cdots \rightarrow R(-1)^{a_1}\rightarrow R^{a_0 }\rightarrow H^0_*(\cE)
\rightarrow 0.
$$     \end{itemize}
 \end{definition}

\begin{theorem}\label{THM} Let $M$ be an indecomposable $b\times a$ matrix with linear entries in $n+1$ variables and constant rank $r$,  Assume that $\cC_M(-1)$ is $L_1$-linear. Then $M$ can be constructed as in Theorem \ref{mainthm}.
\end{theorem}
\begin{proof} Recall the standard  diagram, defined by $M$,
\begin{equation}\label{diagTHinverse}
\xymatrix{
0 \ar[r] & \cK_M \ar[r] &  \cO_{\PP^n}^{ a}\ar[rr]^M\ar[rd] &
& \cO_{\PP^n}(1)^{ b} \ar[r] & \cC_M \ar[r] & 0 . \\
& & & \cE_M\ar[ru]& && \\
}\end{equation}

 \noindent Since $H^i(\PP^n,\cO _{\PP^n}(t))=0$ for any integer $t$ and $0<i<n$, composing the boundary maps $H^0(\cC_M(t))\lra H^1(\cE(t))$ and $H^1(\cE(t))\lra H^2(\cK_M(t))$, in the associated long exact sequences of cohomology, we obtain, for any $t\geq -2$ (for any integer $t$ if $n\geq 3$), an epimorphism $$\mu_t:H^0(\PP^n,\cC_M(t))\lra H^2(\PP^n,\cK_M(t)).$$
We distinguish two cases:

\vskip 2mm
\noindent {\bf Case 1:} $H^2(\PP^n,\cK_M(-1))\ne 0. $ We consider minimal free $R$-resolutions $C_{\bullet }$ and $K_{\bullet}$ of the  truncated graded $R$-modules $[H^0_*(\cC_M(-1))]_{\ge 0}$  and $[H^2_*(\cK_M(-1))]_{\ge 0}$ as well as the morphism of complexes $\widetilde{\mu _{\bullet }} :C_{\bullet }\lra K_{\bullet} $ induced by $\mu _t$. Since $\cC_M(-1)$ is 1-linear and $\widetilde{\mu _t}$ is surjective for any integer $t$, we get that the first step of the resolution of $[H^2_*(\cK_M(-1))]_{\ge 0}$ is also linear.

Using the exact sequence (\ref{diagTHinverse}) we obtain
$$ h^0(\cC_M(-1))=b+h^2(\cK_M(-1))\:\: \text{ and } \:\: h^0(\cC_M)=(n+1)b-a+h^2(\cK_M).$$
Denoting $x:=h^2(\cK_M(-1))$ and $y:=h^2(\cK_M)$, we have the commutative diagram of graded $R$-modules

$$
\xymatrix{
& \vdots \ar[d] & \vdots \ar[d] & \vdots \ar[d] \\
0 \ar[r] & R(-1)^a \ar[d] \ar[r] & R(-1)^{(n+1)x+a-y} \ar[d] \ar[r] & R(-1)^{(n+1)x-y}\ar[d] \ar[r] & 0\\
0 \ar[r] & R^b \ar[r]   & R^{b+x} \ar[r] \ar[d] & R^x \ar[r] \ar[d] & 0\\
 &  & [H^0_*(\cC_M(-1))]_{\ge 0} \ar[d] \ar[r] & [H^2_*(\cK_M(-1))]_{\ge 0} \ar[r] \ar[d] & 0\\
&  & 0 & 0
}
$$
\noindent By sheafifying the above diagram, we obtain the desired result.
\vskip 2mm
\noindent {\bf Case 2:} $H^2(\PP^n,\cK_M(-1))=0. $  In this case $[H^2_*(\cK_M(-1))]_{\ge 0}=0$ and the result  immediately   follows.
\end{proof}

\begin{example}
    (i) As a first application of the previous result, we recover Example \ref{Firstexamples}-(ii). Consider the $14\times 14$ indecomposable matrix $M$ with entries  linear forms in 6 variables and constant rank 9 given by Landsberg and Manivel in  \cite[Example 6.2]{LM}. This matrix sits in the following exact sequence:
\begin{equation}\label{LM}
0
\lra \cK_M \lra \cO_{\PP^5}^{14} \xrightarrow{\,\,\,M\,\,} 
 \cO_{\PP^5}(1)^{14} \lra \Omega ^1(3) \lra 0  
\end{equation}

\noindent where $\cK_M:=ker(\Omega^3(3)\rightarrow \Omega^1(1))$. Being $\Omega^1(1)$ a 0-regular vector bundle, it satisfies the hypotheses of Theorem \ref{THM}. We can therefore recover the previous exact sequence applying Theorem \ref{mainthm} to the complexes $P_{\bullet}$ and $  Q_{\bullet}$, associated respectively to $H^0_*(\Omega ^1(2))_{\ge 0}$ and $H^2_*(\cK_M(-1))_{\ge 0}$, jointly with  the natural  epimorphism $H^0_*(\Omega ^1(2))_{\ge 0}\lra H^2_*(\cK_M(-1))_{\ge 0}$ coming from the 2-term extension depicted in (\ref{LM}).
\vskip 2mm
    (ii)    Let $\cF$ be an instanton bundle on $\PP^3$ with charge  4. In \cite[Theorem 6.1]{BFM}, Boralevi, Faenzi and Mezzetti construct an indecomposable  $14\times 14 $ matrix $M$ of linear forms and constant rank 12 that fits in the diagram \begin{equation}\label{instantondiagram}    
\xymatrix{
0 \ar[r] & \cF (-4) \ar[r] &  \cO_{\PP^3}^{14}\ar[rr]^M\ar[rd] &
&  \cO_{\PP^3}(1)^{14} \ar[r] & \cF(3) \ar[r] & 0  \\
& & & \cE_M\ar[ru]& && \\
}
\end{equation}
    In particular, this can be interpreted as a nontrivial 2-term extension $e\in \Ext^2(\cF(3),\cF(-4))$. Moreover, it corresponds to the matrix presented by Westwick in \cite{W2} without any geometric explanation. Once again, the hypotheses of Theorem \ref{THM} are satisfied in this case, confirming that it can be constructed using the method introduced in Theorem \ref{mainthm}. 
    
By composing the boundary maps in the cohomology long exact sequence of (\ref{instantondiagram}), we obtain, for any integer $t$, a surjective morphism 
$$
H^0(\PP^3,\cF(3+t))\longrightarrow H^2(\PP^3,\cF(-4+t))
$$ 
which induces an epimorphism of  graded $R$-modules 
$$\varphi: H^0_*(\cF)\longrightarrow H^2_*(\cF(-7)).$$
The matrix $M$ can be recovered by considering 
$P_{\bullet }$ and $Q_{\bullet }$ as the sheafifications of the minimal free $R$-resolutions of the truncated graded $R$-modules $H^0_*(\cF)_{\ge 3}$ and $H^2_*(\cF(-7))_{\ge 3}$, respectively. The chain complex map is then naturally induced by $\varphi$.\\
To construct the required resolutions, recall that, by \cite[Th\'eor\'eme 1.1]{R}, a general instanton $\cF$ on $\PP^3$ with charge 4 has the following  resolution
$$
0\rightarrow  \cO (-5)^{4}\rightarrow  \cO (-4)^{10}\rightarrow  \cO (-2)^{4}\oplus \cO (-3)^{4}\rightarrow \cF \rightarrow 0
$$
and its second cohomology group $H^2_*(\cF)$ has the following minimal graded free $R$-resolution (see \cite[Proposition 1]{D})
$$
0\longrightarrow R(-1)^{4}\longrightarrow  R ^{10}\longrightarrow  R(2)^{ 4}\oplus R(3)^{4}\oplus R(1)^{4}\longrightarrow  $$
$$R(4)^{10} \longrightarrow R(5)^{4}\longrightarrow H^2_*(\cF)\cong \Ext^4_R(H^1_*(\cF),R) \longrightarrow 0.$$
\end{example}

\section{Applications}\label{sec-applications}

In this section we focus on the particular case in which the construction described in Theorem \ref{mainthm} is defined by a contraction. This specific situation provides, as detailed in the introduction, several connections with different significant research topics and, in our opinion, represents a first step to study the questions proposed in the final section of this work.

Furthermore, it allows us to:
\begin{itemize}
    \item give a complete  list of the indecomposable matrices of constant rank 7 (see Section \ref{sec-rank7});
    \item describe a method to construct matrices with linear entries and constant rank starting from regular vector bundles (see Section \ref{sec-regularity}).
\end{itemize}
The first point strongly relies on the study of the relations between the classical complexes associated to standard determinantal ideals. 
That is why, in Section \ref{sec-classicalcomplexes}, we recall them for the sake of completeness (see \cite[Theorem A2.10]{Ei} for more details) and we also describe explicit examples that use them as main ingredient.

Regarding regularity, notice that the choice of starting with regular vector bundles is extremely natural since matrices with linear entries and  constant rank already arise as the matrices associated to the morphisms defining the linear resolution of a regular vector bundle. The subsequent and crucial step is to determine if it is possible to choose a vector subspace $W_0 \subset H^0(\cF)$ that still generates the vector bundle, i.e., $eval:W_0 \otimes \cO_{\PP^n} \longrightarrow \cF$ is surjective, 
and induces a commutative diagram that satisfies the hypotheses of Theorem \ref{mainthm}.

Throughout this section, we will suppose all the contractions to be of maximal rank. In particular, they are so in all the detailed examples, because of Example \ref{ex-maxrank} or by a direct computation made with Macaulay2.

\subsection{Classical complexes}\label{sec-classicalcomplexes}
Let  $R$ be a commutative ring and $N$ a free $R$-module. For any element $x \in N$, we construct its associated \textit{Koszul complex} $\eth_p:\bigwedge ^pN\lra \bigwedge ^{p-1}N$  as the dual of the complex  $\delta _q: \bigwedge^{q}N^* \longrightarrow \bigwedge^{q-1}N^*$ whose differentials $\delta _q$ are the contraction by $x$, i.e.,
 $$\delta _q(n_1\wedge \cdots \wedge n_q)=\sum _{i=1}^q (-1)^{i+1}n_i(x)n_1\wedge \cdots \wedge \hat{n_i}\wedge \cdots \wedge n_q.
 $$

Furthermore, given a morphism $\varphi : F \lra G$ between two free $R$-modules $F$ and $G$ of ranks $f$ and $g$,  respectively, we can construct complexes that arise as the ``strands'' of a Koszul complex and we obtain the following two resolutions:
$$
0 \rightarrow S_{f-g}G^* \otimes \bigwedge^{f} F \rightarrow S_{f-g-1}G^* \otimes \bigwedge^{f-1} F \rightarrow \cdots \rightarrow G^* \otimes \bigwedge^{g+1} F \rightarrow \bigwedge^g F \rightarrow R 
$$
known as the \textit{Eagon-Northcott complex}, and
$$
0 \rightarrow S_{f-g-1}G^* \otimes \bigwedge^{f} F \rightarrow S_{f-g-2}G^* \otimes \bigwedge^{f-1} F \rightarrow \cdots \rightarrow G^* \otimes \bigwedge^{g+2} F \rightarrow \bigwedge^{g+1} F \rightarrow F \rightarrow G 
$$
known as the \textit{Buchsbaum-Rim complex}. More in general, we have: 

$$
0 \rightarrow S_{f-g-i}G^* \otimes \bigwedge^{f} F \rightarrow S_{f-g-i-1}G^* \otimes \bigwedge^{f-1} F \rightarrow \cdots $$
$$\rightarrow G^* \otimes \bigwedge^{g+i} F \rightarrow \bigwedge^{i} F \rightarrow G\otimes \bigwedge ^{i-1}F \rightarrow \cdots \rightarrow S_{i-1}(G)\otimes F \rightarrow S_i( G).
$$

A clear description of the differentials in the complexes will help us to demonstrate the required commutativity in our construction. Therefore, by fixing sets of generators $\{f_i\}_{i=0}^{f-1}$ for $F$ and $\{g_j\}_{j=0}^{g-1}$ for $G$, and denoting the induced ones in  $G^*$ by $\{g_j^*\}_{j=0}^{g-1}$, we can describe 
\begin{equation}\label{def_delta}
\begin{array}{rccc}
\delta_{d,p}: & S_{d}G^* \otimes \bigwedge^{p} F & \longrightarrow & S_{d-1}G^* \otimes \bigwedge^{p-1} F\\
& (g^*_{j_1}\cdots g_{j_d}^*)\otimes f & \mapsto & \sum_{k=1}^d \left(g_{j_k}((g^*_{j_1}\cdots g_{j_d}^*))\right)\otimes\left(\varphi^*(g_{j_k}^*)(f)\right) 
\end{array}
\end{equation}

and 
\begin{equation}\label{def_gamma}
\begin{array}{rccc}
\gamma_{j,i}: & S_{j}G \otimes \bigwedge^{i} F & \longrightarrow & S_{j+1}G \otimes \bigwedge^{i-1} F\\
& m\otimes f & \mapsto & \sum_{i=0}^{g-1} g_im\otimes \varphi^*(g_i^*)(f). 
\end{array}
\end{equation}

\vskip 2mm
Notice that $\varphi^*(g_{j_k}^*)\in F^*$ acts on $\bigwedge F$.

\begin{example}\label{Ex-irrelevantideal}
If $R=k[x_0,\ldots,x_n]$ and $\varphi: F=R(-1)^f \longrightarrow G=R^g$ is defined by
\begin{equation}\label{mat-powerm}  
\left[
\begin{array}{ccccccccccccccccc}
x_0 & x_1 & \cdots & x_n & 0 & \cdots & 0 & 0 & 0\\
0 & x_0 & x_1 & \cdots & x_n & 0 & \cdots & 0 & 0\\
\vdots & & \ddots & & \ddots & & & & \vdots\\
0 & \cdots & 0 & 0 & x_0 & x_1 & \cdots & x_n & 0 \\
 0 & \cdots & 0 & 0 & 0 & x_0 & x_1 & \cdots & x_n 
\end{array}
\right],
\end{equation}
the Eagon-Northcott complex provides a resolution of the $g$-th power of the maximal   ideal $m=(x_0,\cdots ,x_n)$ of $R$.   
\end{example}

\subsubsection{Relating Koszul complexes} From now on, we   work either with the above complexes of graded $R$-modules or with the complexes in $\PP^n$ given by their sheafification.
Examples of complexes that satisfy Theorem \ref{mainthm} naturally show up  when we consider the following commutative diagram, which we define for vector bundles in $\PP^n$,
\begin{equation}\label{diagram1} 
\xymatrix{
\bigwedge^{p+1} V^* \otimes \cO_{\PP^n}(-1) \ar[r]^{-\wedge \omega} \ar[d]_\eth & \bigwedge^{q+1} V^* \otimes \cO_{\PP^n}(-1) \ar[d]^\eth\\
\bigwedge^{p} V^* \otimes \cO_{\PP^n} \ar[r]^{-\wedge \omega} & \bigwedge^{q} V^* \otimes \cO_{\PP^n}
}
\end{equation}
where $p \geq q$, $\eth$ denotes the Koszul derivation and $-\wedge \omega$ the contraction by a fixed element $\omega \in \bigwedge^{p-q}V$.

Diagram (\ref{diagram1}) induces a map
$\Omega^{p}_{\PP^n}(p) \longrightarrow \Omega^{q}_{\PP^n}(q)$. Indeed,  using the Koszul complex, we obtain
$$
\Hom(\Omega_{\PP^n}^{p}(p), \Omega_{\PP^n}^q(q))=H^0(\PP^n,(\bigwedge^{p-q}T_{\PP^n})(p-q))=\bigwedge^{p-q}V,
$$
which is again defined by contraction.
The composition of morphisms corresponds to multiplication in the exterior algebra $\bigwedge V$.

\begin{remark}
  Diagram (\ref{diagram1}) can be naturally generalized by fixing the appropriate number of elements in $\bigwedge^{p-q}V$, leading to a ``contraction''
    \begin{equation}\label{contraction-wedge}
    - \wedge \omega_{(s\times t)} : (\bigwedge^{p}V^*)^{t} \longrightarrow (\bigwedge^{q}V^*)^{s}
    \end{equation}
    that induces a morphism
  $$  - \wedge \widetilde{\omega}_{(s\times t)} : \left(\Omega_{\PP^n}^p(p)\right)^{t} \longrightarrow \left(\Omega_{\PP^n}^q(q)\right)^{s}.
    $$
  
    Hence we can associate, to the previous vector bundle morphism, a $s\times t$ matrix $A$ with entries in $\bigwedge^{p-q}V$.
    It holds that: $- \wedge \widetilde{\omega}_{(s\times t)}$ is pointwise surjective if and only if its dual map  $\left(- \wedge \widetilde{\omega}_{(s\times t)}\right)^*$ is pointwise injective if and only if  
    $$\wedge A: \left(\bigwedge ^{q}V\wedge \omega\right)^{s} \longrightarrow \left(\bigwedge ^{p}V\wedge \omega\right)^{t}$$ is injective for any point
      $\langle v\rangle \in \PP(V)$. 

      Furthermore, for any choice of the elements in $\bigwedge^{p-q}V$ that define a maximal rank contraction we have, whenever the dimension allows it, i.e., when
      $$
      t\binom{n+1}{p}\geq s\binom{n+1}{q},
      $$
      a surjetive map as required.

\end{remark}

\subsubsection{Strands of the Koszul complex}\label{sec-strands}

Other families of examples appear when considering different strands of the Koszul complex (or the same strand properly shifted).

Consider $R$ and $\varphi$ defined as in Example \ref{Ex-irrelevantideal}.
The map between the sheafified complexes, needed in Theorem \ref{mainthm}, can be made explicit through the following commutative diagrams. For simplicity, we will use the same notation for the vector spaces involved as for their corresponding modules.

First of all, we describe a diagram involving the ``left part'' of each strand, and therefore we consider
$$
\xymatrix{
S_{d}G^* \otimes \bigwedge^{p} F \otimes \cO_{\PP^n}(-1) \ar[d]_{\delta_{d,p}} \ar[r]^{\pi_{(d,b)}^{(p,q)}} & S_{b}G^* \otimes \bigwedge^{q} F \otimes \cO_{\PP^n}(-1) \ar[d]^{\delta_{b,q}} \\ 
S_{d-1}G^* \otimes \bigwedge^{p-1} F \otimes \cO_{\PP^n} \ar[r]^{\pi_{(d-1,b-1)}^{(p-1,q-1)}} & S_{b-1}G^* \otimes \bigwedge^{q-1} F \otimes \cO_{\PP^n}
}
$$

\noindent for $p\geq q$, $d\geq b$ and $\delta_{d,p}$ denotes the derivations that appeared in the strand complexes. Observe that, in this case, we can describe the images by $\delta_{d,p}$ (which are defined in Formula (\ref{def_delta})) as
\begin{equation}\label{derivations-strand}
\begin{array}{cc}
\delta_{d,p}\left((g^*_{i_1}\cdots g^*_{i_d})\otimes (f_{j_1}\wedge \ldots \wedge f_{j_p})\right) =  \vspace{2mm}\\
\sum_{k=1}^d\left(\frac{g^*_{i_1}\cdots g^*_{i_d}}{g^*_{i_k}}\otimes \left(\sum_{t=1}^p (-1)^{t-1} x_{j_t-i_k}(f_{j_1}\wedge \ldots \wedge \hat{f_{j_t}}\wedge \cdots \wedge  f_{j_p})\right)\right).
\end{array}
\end{equation}

The maps $\pi_{(d,b)}^{(p,q)}$ are defined composing
$$
id \otimes (-\wedge \omega): S_{d}G^* \otimes \bigwedge^{p} F \otimes \cO_{\PP^n}(-1)  \longrightarrow S_{d}G^* \otimes \bigwedge^{q} F \otimes \cO_{\PP^n}(-1),
$$
constructed using the contraction given by a fixed element $\omega \in \bigwedge^{p-q}F^*$, with
\begin{equation}\label{proj-symmetric}
\begin{array}{rcccccccc}
\pi_{(d,b)}^G \otimes id : & S_{d}G^* \otimes \bigwedge^{q} F & \longrightarrow & S_{b}G^* \otimes \bigwedge^{q} F\\
& g^*_{i_1}\cdots g^*_{i_d} \otimes (f_{j_1}\wedge \cdots \wedge f_{j_q}) & \mapsto & \left(\sum_{\Delta \subset \{i_1,\ldots,i_d\}}g_\Delta\right) \otimes (f_{j_1}\wedge \cdots \wedge f_{j_q})
\end{array}
\end{equation}
where $\Delta$ varies to obtain all possible monomials $g_\Delta$, of degree $b$, from the set $\{g^*_{i_1},\ldots, g^*_{i_d}\}$.

Analogously, given the definition of $\gamma_{d,p}$ in Formula (\ref{def_gamma}), we can describe a diagram involving the ``right part'' of each strand, and describe
$$
\xymatrix{
S_{d}G \otimes \bigwedge^{p} F \otimes \cO_{\PP^n}(-1) \ar[d]_{\gamma_{d,p}} \ar[r]^{\pi_{(d,b)}^{(p,q)}} & S_{b}G \otimes \bigwedge^{q} F \otimes \cO_{\PP^n}(-1) \ar[d]^{\gamma_{b,q}} \\ 
S_{d+1}G \otimes \bigwedge^{p-1} F \otimes \cO_{\PP^n} \ar[r]^{\pi_{(d-1,b-1)}^{(p-1,q-1)}} & S_{b+1}G \otimes \bigwedge^{q-1} F \otimes \cO_{\PP^n}
}
$$
\begin{equation}\label{derivations-strand2}
\begin{array}{cc}
\gamma_{d,p}\left((g_{i_1}\cdots g_{i_d})\otimes (f_{j_1}\wedge \ldots \wedge f_{j_p})\right) =  \vspace{2mm}\\
\sum_{k=0}^{g-1}\left(g_{i_1}\cdots g_{i_d} \cdot g_k \otimes \left(\sum_{t=1}^p (-1)^{t-1} x_{j_t-k}(f_{j_1}\wedge \ldots \wedge \hat{f_{j_t}}\wedge \cdots \wedge  f_{j_p})\right)\right).
\end{array}
\end{equation}

\begin{remark}
The contraction used to define the derivation $\delta_{p,q}$ in the strand complex is a left one (see Definition \ref{def-contraction}). On the other hand, the maps $\pi_{(d,b)}^{(p,q)}$ are defined using a right contraction. This choice allows us to obtain the required commutative diagrams without the need to account for an alternate change of sign when defining $\pi_{(d,b)}^{(p,q)}$ . 
\end{remark}

\subsubsection{Mixed relations}
Another family of examples arises from commutative diagrams that relate the (sheafification of the) Koszul and Eagon-Northcott complexes. In particular, we will see in the following examples that such construction will also  be related to the so-called \textit{Drézet bundles} on $\PP^n$.

\begin{definition}
A  Drézet bundle on $\PP^n$ is a rank $r$ vector bundle $\cE$ fitting into a short exact sequence:
$$
0\lra \cO _{\PP^n}(-d)\lra \cO_{\PP^n}^{ r+1} \lra \cE_d \lra 0.
$$
\end{definition}
They are known to be stable and therefore simple. Moreover, the dual $\cE^{\vee}$ of a  Drézet bundle is a special case of the so-called syzygy bundle (\cite{FM}).

Consider again the  ring $R$ and $\varphi$ as in Example \ref{Ex-irrelevantideal}.
The map between the sheafified complexes, needed in Theorem \ref{mainthm}, is described by the commutative diagram that follows.\\
We consider
$$
\xymatrix{
S_{d}G^* \otimes \bigwedge^{p} F \otimes \cO_{\PP^n}(-1) \ar[d]_{\delta_{d,p}} \ar[r]^-{h_{(d,p,q)}^\omega} & \bigwedge^q V^* \otimes \cO_{\PP^n}(-1) \ar[d]^{\eth} \\ 
S_{d-1}G^* \otimes \bigwedge^{p-1} F \otimes \cO_{\PP^n} \ar[r]^-{h_{(d-1,p-1,q-1)}^\omega} & \bigwedge^{q-1} V^* \otimes \cO_{\PP^n}
}
$$
where $p\geq q$, $d\geq q$ and $\delta_{d,p}$, $\eth$ denote, respectively, the derivations of the Eagon-Northcott and the Koszul complexes. Recall that the explicit description of the images by $\delta_{d,p}$ is given by Formula (\ref{derivations-strand}).

The maps $h_{(d,p,q)}^\omega$ are constructed combining $(id \otimes (-\wedge \omega))$, which is induced by a maximal rank contraction by a fixed element $\omega \in \bigwedge^{p-q}F^*$,
with 
$$
\xymatrix{
S_{d}G^* \otimes \bigwedge^{q} F \otimes \cO_{\PP^n}(-1) \ar[d]_{\delta_{d,q}} \ar[r]^-{h_{(d,q)}} & \bigwedge^q V^* \otimes \cO_{\PP^n}(-1) \ar[d]^{\eth} \\ 
S_{d-1}G^* \otimes \bigwedge^{q-1} F \otimes \cO_{\PP^n} \ar[r]^-{h_{(d-1,q-1)}} & \bigwedge^{q-1} V^* \otimes \cO_{\PP^n}
}
$$
When $d=q$, we set
$$
h_{q,q}\left( (g^*_{i_1}\cdots g^*_{i_q})\otimes (f_{j_1}\wedge \ldots \wedge f_{j_q})\right) = \sum_{\sigma \in \mathfrak{S}_q}\left(v^*_{j_1 - \sigma(i_1)} \wedge \cdots \wedge v^*_{j_q - \sigma(i_q)}\right)
$$
using the convention that $v^*_{j-i} = 0$ if $0 > j-i > n$. 

For any $d \leq q$ we define $h_{(d,q)} = h_{(q,q)} \circ \pi_{(d,q)}^G$, where $\pi_{(d,q)}^G$ is defined as in (\ref{proj-symmetric}).

\begin{remark}    
The described construction can be generalized if we consider copies of the complexes involved. Indeed, fixing elements $\omega_{(1,1)}, \ldots, \omega_{(t,s)} \in \bigwedge^{p-q}F^*$, it is possible to define
$$
\left(S_{d}G^* \otimes \bigwedge^{q} F\right)^{t} \otimes \cO_{\PP^n} \stackrel{h_{(d,q,t,s)}^{\omega_{(i,j)}}}{\longrightarrow}  \left(\bigwedge^q V^*\right)^{s} \otimes \cO_{\PP^n}
$$
and, for a sufficiently independent choice of them combined with a proper dimensional range, they will satisfy the hypotheses of Theorem \ref{mainthm}.
\end{remark}

The following proposition tells us how many monomials of degree $d$ can be excluded to still obtain a $(-q,-1)$-linear resolution.
More concretely, for which values of $\alpha$ we have a commutative diagram
\begin{equation}\label{diag-monomialcut}
\xymatrix{
 & 0 \ar[d] & 0 \ar[d] &  \\
 0 \ar[r] & \cF \ar[d] \ar[r] & \cE^\lor_d \ar[r] \ar[d] & \cO_{\PP^n}^\alpha \ar[d]^\simeq \ar[r] & 0\\
 0 \ar[r] & \cO_{\PP^n}^{\binom{n+d}{d} -\alpha} \ar[d] \ar[r] & \cO_{\PP^n}^{\binom{n+d}{d}} \ar[r] \ar[d] & \cO_{\PP^n}^\alpha \ar[r] & 0 \\
 & \cO_{\PP^n}(d) \ar[r]^{\simeq} \ar[d] & \cO_{\PP^n}(d) \ar[d]\\
 & 0 & 0
}
\end{equation}
that leads to a resolution of type
$$
\cdots \rightarrow P_{-q-1} \rightarrow P_{-q} \rightarrow \cdots \rightarrow P_{-1} \rightarrow \cF \rightarrow 0
$$
with $P_i = \cO_{\PP^n}(i)^{p_i}$ for $-q\leq i \leq -1$. This resolution translates into having families of  matrices with linear entries and constant rank.

\begin{proposition} \label{drezet} 
    Given integers $d\geq 2$,  $n\ge 2$ and   $s>0$, we consider the maximal $q \geq 2$ such that the generic elements $\omega_1,\ldots,\omega_s \in \bigwedge^{d}F^*$ define a maximal rank contraction and the relation 
    \begin{equation}\label{copies-surj} 
    \binom{d+q-1}{q} \cdot \binom{n+d}{q+d} \geq s \binom{n+1}{q}
    \end{equation}
    is satisfied. Then, for any $i = 2,\ldots,q$, we have  an indecomposable matrix $M_{i-1}$, of dimension
    $$
    \left(\binom{d+i-3}{i-2}\binom{n+d}{i-2+d} -s \binom{n+1}{i-2}\right) \times \left(\binom{d+i-2}{i-1}\binom{n+d}{i-1+d} -s \binom{n+1}{i-1}\right)
    $$
    of linear entries and constant rank 
    $$
    \left(\sum_{k=0}^{i-2} (-1)^k \left(\binom{d+i-3-k}{i-2-k}\binom{n+d}{i-2+d-k} -s \binom{n+1}{i-2-k}\right)\right) + (-1)^{i-1}.
    $$
\end{proposition}
\begin{proof}
The inequality described in (\ref{copies-surj}) ensures the surjectivity of the map
\begin{equation}\label{drezet-surj}
h_{(q,d+q,q)}^{\{\omega_j\}_{j=1}^s} : S_q G^* \otimes \bigwedge^{q+d} F \otimes \cO_{\PP^n}(-q) \longrightarrow \left(\bigwedge^q V^* \otimes \cO_{\PP^n}(-d) \right)^{ s}
\end{equation}
constructed taking  $\omega_1,\ldots,\omega_s \in \bigwedge^{d}F^*$. Furthermore, if (\ref{copies-surj}) holds, then
$$
 \binom{d+q-1-t}{q-t} \cdot \binom{n+d}{q+d-t} \geq s \binom{n+1}{q-t}
 $$
 holds as well for any $i=0,\ldots,q$. This means that we get surjective maps $H_{q-j}:= h_{(q-j,d+q-j,q-j)}^{\{\omega_j\}_{j=1}^s}$, defined analogously as in (\ref{drezet-surj}) for any $i=0,\ldots,q$, that provide the following commutative diagram ($\cO=\cO_{\PP^n}$) 
 $$
 \xymatrix{
W_i \otimes \cO(-i) \ar[r] \ar[d] &S_i G^* \otimes \bigwedge^{i+d} F \otimes \cO(-i) \ar[r]^-{H_i} \ar[d] & \left(\bigwedge^i V^* \otimes \cO(-i) \right)^{ s}\ar[d]\\
W_{i-1} \otimes \cO(-i+1) \ar[r] \ar[d]_{M_{i-1}} &S_{i-1} G^* \otimes \bigwedge^{i-1+d} F \otimes \cO(-i+1) \ar[r]^-{H_{i-1}} \ar[d] & \left(\bigwedge^{i-1} V^* \otimes \cO(-i+1) \right)^{ s}\ar[d]\\
W_{i-2} \otimes \otimes \cO(-i+2) \ar[r]  &S_{i-2} G^* \otimes \bigwedge^{i-2+d} F \otimes \cO(-i+2) \ar[r]^-{H_{i-2}} & \left(\bigwedge^{i-2} V^* \otimes \cO(-i+2) \right)^{ s}\\
 }
 $$
 for any $i=2,\ldots,q$.
 
 Once again, Theorem \ref{mainthm} ensures that $M_{i-1}$ is a matrix, of linear entries and constant rank, of the required dimension. Finally, the rank is computed iteratively starting from the rank of $M_1$ which is equal to $\binom{n+d}{d} - s.$ 
\end{proof}

\begin{remark}
(1) 
The proposed technique strongly relies on the choice of an element $\omega \in \bigwedge^{p-q} F^*$, that induces a contraction defined as
\begin{equation}\label{contraction-pq}
S_{d}G^* \otimes \bigwedge^{p} F  \stackrel{id \otimes (-\wedge \omega)}{\longrightarrow} S_{d}G^* \otimes \bigwedge^{q} F.
\end{equation}
These maps are in a one to one correspondence with maps over the vector bundles obtained from the respective splitting of the two complexes. In fact, let
$$
K_{d,p}:= \ker \left( S_{d}G^* \otimes \bigwedge^{p} F \otimes \cO_{\PP^n}(-d)\stackrel{\delta_{d,p}}{\longrightarrow}S_{d-1}G^* \otimes \bigwedge^{p-1} F \otimes \cO_{\PP^n}(-d+1)\right).
$$
Denote by $T$ the elements in $\Hom\left(K_{d,p},K_{d,q}\right)$ that induce the identity at the $S_d G^*$ component of the tensor product.
We have that $T \simeq \bigwedge^{p-q} F^*$. This can be proven by induction considering the short exact sequences
$$
0 \longrightarrow K_{d+1,p+1}(d) \longrightarrow S_{d+1}G^* \otimes \bigwedge^{p+1} F \otimes \cO_{\PP^n}(-1) \longrightarrow K_{d,p}(d)\longrightarrow 0
$$
and
$$
0 \longrightarrow K_{d+1,q+1}(d) \longrightarrow S_{d+1}G^* \otimes \bigwedge^{q+1} F \otimes \cO_{\PP^n}(-1) \longrightarrow K_{d,q}(d) \longrightarrow 0 .
$$
Indeed, applying $\Hom\left(K_{d,p}, - \right)$ to the second short exact sequence, followed by applying $\Hom\left(-, K_{d+1,q+1}\right)$ to the first one, we get a chain of isomorphism
$$
\Hom\left(K_{d,p}, K_{d,q}\right) \simeq \Ext^1 \left(K_{d,p}, K_{d+1,q+1} \right) \simeq \Hom\left(K_{d+1,p+1}, K_{d+1,q+1}\right).
$$
Iterating the described process, we obtain the equivalence
$$
\begin{array}{rl}
\Hom\left(K_{d,p}, K_{d,q}\right) & \simeq \Hom\left(K_{d+f-p,f}, K_{d+f-p,q+f-p}\right) \vspace{1mm}\\
& \simeq \Hom\left(S_{d+f-p}G^*\otimes \cO_{\PP^n}(d+f-p), K_{d+f-p,q+f-p}\right)\vspace{1mm}\\
& \simeq S_{d+f-p}G\otimes S_{d+f-p}G^*\otimes \bigwedge^{p-q} F^*,
\end{array}
$$
which implies the claimed description of $T$.

In turn, this implies that the map described in (\ref{contraction-pq}) is equivalent to considering the corresponding one
$$
id \otimes (- \wedge \omega) : K_{d,p} \longrightarrow K_{d,q}.
$$
Furthermore, $id \otimes (-\wedge \omega) : S_{d}G^* \otimes \bigwedge^{p} F  \longrightarrow S_{d}G^* \otimes \bigwedge^{q} F
$ is surjective if and only if 
$id \otimes (- \wedge \omega) : K_{d,p} \longrightarrow K_{d,q}
$ is so. Hence, the bounds in Proposition \ref{drezet} are sharp.

(2) In light of the previous construction, we observe that there exists a vector bundle $\cE_d$ with $c_1(\cE_d) =d$, generated by global sections and given as the image of a matrix of linear forms and constant rank, for any $d\geq 0$. Indeed, it is sufficient to consider $\cF^\lor(1)$,  with $\cF$ defined in Diagram (\ref{diag-monomialcut}). Therefore, in order to ponder a feasible classification problem, it is natural to impose a bound on the first Chern class and the rank of the vector bundle, or equivalently bounds on both first Chern classes of $\cE_d$ and $\cE_d^\lor(1)$, as done in \cite{MMR}. 
\end{remark}

Let us illustrate Proposition \ref{drezet} with a couple of concrete examples.

\begin{example}\label{ex-14}
    (1) For $n=3$, $f=5$, $g=2$, $t=1$ and $s=3$ we recover the Drézet bundle on $\PP^3$ defined by a net of quadrics, given by Manivel and Miró-Roig in \cite[Theorem 2(2)]{MMR} and, in particular,
    the $12\times 8$ indecomposable matrix with linear entries in $k[x_0,x_1,x_2,x_3]$ and constant rank 7 in \cite[Example 29]{MMR}. 

    \vskip 2mm
    (2) In line with the previous example, we will present this one in detail to offer a clearer understanding of the proposed construction.\\ 
    Let $n=3$, $f=6$, $g=3$, $t=1$ and $s=5$ (observe that an analogous process applies for any $s$, $1\le s\le 6$ by Proposition \ref{drezet}). 
    In this case, $\varphi : F \rightarrow G$ is represented by the matrix
    $$
    \left[
    \begin{array}{ccccccccccc}
    x_0 & x_1 & x_2 & x_3 & 0 & 0\\
    0 & x_0 & x_1 & x_2 & x_3 & 0\\
    0 & 0 & x_0 & x_1 & x_2 & x_3 
    \end{array}
    \right].
    $$
    We now consider  five generic elements $\omega _1, \omega _2, \omega _3, \omega _4, \omega _5 \in \bigwedge ^3 F^*$ that define maximal rank contractions
    $$
    \bigwedge^p F \longrightarrow \bigwedge^{p-3} F
    $$
    and the commutative diagram
    $$
    \xymatrix{
    S_2 G^* \otimes \bigwedge^2 F \otimes \cO_{\PP^3}(-2) \ar[d]^{\delta_{2,2}} \ar[r]^-{h_{2,2}} & \bigwedge^2 V^* \otimes \cO_{\PP^3}(-2)\ar[d]^\eth \\
    G^* \otimes F \otimes \cO_{\PP^3}(-1)\ar[d]^{\delta_{1,1}} \ar[r]^-{h_{1,1}} & V^* \otimes \cO_{\PP^3}(-1) \ar[d]^\eth\\
    \cO_{\PP^3} \ar[r]^-{h_{0,0}} & \cO_{\PP^3}
    }
    $$
    where, denoting by $\{f_i\}_{i=0}^5$, $\{g_j\}_{j=0}^2$ and $\{v_k\}_{k=0}^3$ the sets of generators of $F$, $G$ and $V$, respectively, we have\\
    - $h_{0,0}$ to be the identity;\vspace{1mm}\\
    - $h_{1,1}(g_j^*\otimes f_i) = v^*_{i-j}$;\vspace{1mm}\\
    - $h_{2,2}((g_j^*g_p^*)\otimes (f_i \wedge f_t)) = v^*_{i-j}\wedge v^*_{p-t} + v^*_{i-p} \wedge v^*_{t-j}$;\vspace{1mm}\\
    recalling that $v^*_{i-j}= 0$ if $i-j<0$ or $i-j > 3$.\vspace{1mm}\\

    Combining the two previous diagrams, we get that all horizontal maps are surjective. Once again,  we are under the hypotheses of Theorem \ref{mainthm} and we have
    
    $$
\xymatrix{
0 \ar[r] & W_2 \otimes \cO_{\PP^3}(-2) \ar[r] \ar[d]_{\mathfrak{d}_2} & S_2 G^* \otimes \bigwedge^5 F \otimes \cO_{\PP^3}(-2) \ar[d] \ar[r]^{H_2} & \left(\bigwedge^2 V^*\right)^{ 5} \otimes \cO_{\PP^3}(-2) \ar[r] \ar[d] & 0 \\
0 \ar[r] & W_1 \otimes \cO_{\PP^3}(-1) \ar[d] \ar[r]^{S_1} &  G^* \otimes \bigwedge^4 F \otimes \cO_{\PP^3}(-1) \ar[d] \ar[r]^{H_1} &  ((V^*)^{ 5}) \otimes \cO_{\PP^3}(-1) \ar[r] \ar[d] & 0 \\
0 \ar[r] & \cE_M(-1) \ar[r] \ar[d] & \mathcal{F} \ar[r] \ar[d] & \cO_{\PP^3}^{ 5} \ar[r] \ar[d] & 0\\
& 0 & 0 & 0
}
$$
where $\cF$ denotes the homogeneous Drézet
bundle defined as kernel of all degree three monomials in $k[x_0,x_1,x_2,x_3]$, i.e.,
$$
0 \longrightarrow \cF \longrightarrow S_3 V \otimes \cO_{\PP^3} \longrightarrow \cO_{\PP^3}(3) \longrightarrow 0
$$
and $\cE_M$ denotes the  rank 14 Drézet bundle on $\PP^3$ that we were looking for. Denoting by $\mathcal{K}_M$ the image of $\mathfrak{d}_2$, this vector  bundle gives rise to the following diagram:

$$\xymatrix{
0 \ar[r] &\cK_M \ar[r] &\cO_{\PP^3}^{ 25}\ar[rr]^M\ar[rd] &
& \cO_{\PP^3}(1)^{ 15} \ar[r] & \cO_{\PP^3}(4) \ar[r] &  0  \\
& & & \cE_M \ar[ru]& && 
}$$
where $M$ is an  indecomposable $25 \times 15$ matrix with linear entries and constant rank 14. An explicit description of how to construct this matrix is given in Appendix \ref{example-appendix}.
\end{example}

\vskip 4mm

\subsection{Indecomposible matrices of constant rank 7}\label{sec-rank7}
This part is dedicated to proving the following result.
\begin{theorem}\label{thm-rank7}
    Given an indecomposable matrix $M$ of linear forms of constant rank 7, the associated vector bundle $\cE_M$ must be, up to switching with $\cE_M^{\lor}(1)$, one of the following types:
    \begin{enumerate}
        \item ($c_1=1$) the twist $T_{\PP^7}(-1)$ of the tangent bundle on $\PP^7$, 
        \item ($c_1=2$) a Drézet bundle on $\PP^3$ defined by the sequence
        $$
        0 \longrightarrow \cO_{\PP^3}(-2) \longrightarrow \cO_{\PP^3}^{8} \longrightarrow \cE_M \longrightarrow 0;
        $$
        \item ($c_1=3$) a Steiner bundle on $\PP^3$ defined by the sequence
        $$
        0 \longrightarrow \cO_{\PP^3}(-1)^{3} \longrightarrow \cO_{\PP^3}^{10} \longrightarrow \cE_M \longrightarrow 0;
        $$
        \item ($c_1=3$) a Drézet bundle on $\PP^2$ defined by the sequence
        $$
        0 \longrightarrow \cO_{\PP^2}(-3) \longrightarrow \cO_{\PP^2}^{8} \longrightarrow \cE_M \longrightarrow 0.
        $$
    \end{enumerate}
\end{theorem}

Notice first that, since the roles of $\cE_M$ and $\cE_M^{\lor}(-1)$ are completely symmetric, and given that $c_1(\cE_M^\lor(-1)) = 7 - c_1(\cE_M)$, we can assume that $2c_1(\cE_M) \leq 7$, i.e., $c_1(\cE_M) \leq 3$.
Moreover, recall that, since both $\cE_M$ and $\cE_M^\lor(1)$ are globally generated, the vector bundle $\cE_M$ is uniform. Therefore, we can suppose that $\rk(\cE_M) \geq n +2$; otherwise $\cE_M$ would be a line bundle, the tangent bundle  or the cotangent bundle. In our case, this gives $n\leq 5$.

For $c_1(\cE_M) = 1$, it is already known that $\cE_M$ can only be $T_{\PP^7}(-1)$ (see \cite[Theorem 2.4]{EH}). For $c_1(\cE_M)=2,3$, based on the results available in \cite{SU} and \cite[Proposition 36]{MMR}, we have:
\begin{proposition}\label{prop-candidatebundle}
Let $\cE_M$ be an indecomposable rank 7 vector bundle on $\PP^n$, with $n \leq 5$ and $c_1(\cE_M)=2,3.$ Suppose that both $\cE_M$ and $\cE_M^\lor(1)$ are globally generated. Then either:
\begin{itemize}
    \item $\cE_M$ is a Steiner bundle, or
    \item $\cE_M$ is a Drézet bundle, or
    \item $\cE_M$ is defined by 
    $$
    0 \longrightarrow \cO_{\PP^n}(-1) \oplus \cO_{\PP^n}(-2) \longrightarrow \cO_{\PP^n}^9 \longrightarrow \cE_M \longrightarrow 0.
    $$
\end{itemize}
\end{proposition}
    To work with the same diagrams described in the previous sections, we will actually use the dual vector bundle $\cE_M^\lor$.
    For each vector bundle $\cE_M^\lor$ coming from the previous list, the associated $\cC_M(-1)$, as defined in Diagram (\ref{maindiagram}), is $L_1$-linear. Hence, thanks to Theorem \ref{THM}, we have:
    
    \begin{corollary}\label{cor-rank7}
    Every family of indecomposable matrices of constant rank 7 can be constructed as in Theorem \ref{mainthm}.
    \end{corollary} 
\begin{remark}
   If we suppose the rank of the indecomposable matrices to be greater or equal than 8, the first Chern class relation determined by the symmetric role of $\cE_M$ and $\cE_M^\lor(1)$ implies that $c_1(\cE_M)$ can be greater or equal then 4. Unfortunately, the classification of globally generated vector bundles with $c_1\ge 4$, which plays a fundamental role in the results provided in \cite{MMR}, is far from being complete. So far, it is known for $c_1(\cE)\leq 3$ and, with the additional hypothesis $H^i(\cE^\lor)=0$ for $i=0,1$, the classification extends to $c_1(\cE)\leq 5$ (see \cite{ACM} and \cite{ACM2}). Nevertheless, all of the possible vector bundles in the list (considering their rank and the property of having $\cE^\lor(1)$ globally generated as well) have, again, an associated $L_1$-linear bundle $\cC_M(-1)$. Therefore, they can be constructed as in Theorem \ref{mainthm}.
\end{remark}
We will now focus on each case appearing in Proposition \ref{prop-candidatebundle}, to determine which ones are given by an indecomposable matrix.\\

\noindent\textbf{Steiner bundles.} Being $\cE_M$ associated to an indecomposable matrix, we have that its splitting type is constant of the form $(0^{a_0},1^{a_1})$, with $a_1>0$.
By \cite[Theorem 3.5]{MarMR}, we know that $c_1(\cE_M)+2(n-1) \leq \rk(\cE_M) \leq c_1(\cE_M)n$. This inequality implies that the only possibility, in our case, is given by $n=3$ and $c_1(\cE_M)=3$.

We consider 6 generic elements $v_1,v_2,v_3,v_4,v_5,v_6 \in V$,  $\dim V =4$, to construct the contractions which define the following commutative diagram
$$
\xymatrix{
 & \vdots \ar[d] & \vdots \ar[d]  & \vdots \ar[d] &  \\
0 \ar[r] & \cO_{\PP^3}(-1)^{10} \ar[d]_M \ar[r] &  \left(\bigwedge^2 V^*\right)^{3} \otimes \cO_{\PP^3}(-1) \ar[d] \ar[r] &  ((V^*)^{2}) \otimes \cO_{\PP^3}(-1) \ar[r] \ar[d] & 0 \\
0 \ar[r] & \cO_{\PP^3}^{10} \ar[r] \ar[d] & ((V^*)^{3}) \otimes \cO_{\PP^3} \ar[r] \ar[d] & \cO_{\PP^3}^{2} \ar[r] \ar[d] & 0\\
& \cO_{\PP^3}(1)^{3} \ar[d] \ar[r]^\simeq & \cO_{\PP^3}(1)^{3} \ar[d] & 0\\
& 0 & 0
}
$$
Using Macaulay 2, with a similar step by step construction given in Appendix \ref{example-appendix}, it is possible to give an explicit $10\times 10$-matrix $M$ of linear forms and constant rank 7, associated to the Steiner bundle
$$
0 \longrightarrow \cE_M^\lor \longrightarrow \cO_{\PP^3}^{10} \longrightarrow \cO_{\PP^3}(1)^{3} \longrightarrow 0.
$$
We have therefore proven that every Steiner bundle is of type (3) in Theorem \ref{thm-rank7}.\\

\noindent\textbf{Drézet bundles.} The rank 7 Drézet bundles in item (2) and (4) of Theorem \ref{thm-rank7} are strictly related to the ones of rank 6 obtained in \cite[Theorem 2]{MMR}. Indeed, the latter  have an associated $L_1$-linear bundle and can be constructed as in Theorem \ref{mainthm}. This means that the rank 7 vector bundles, and hence the corresponding rank 7 indecomposable matrices, can be obtained contracting by one vector less.

We only have to prove that all the other possible cases cannot occur.\\
If $c_1=2$ and $n=2$, the vector bundle $\cE_M$ is forced to be the direct sum of the unique homogeneous rank 5 Drézet bundle with two copies of the trivial bundle and, hence, it is decomposable. The remaining cases can be excluded as follows. Corollary \ref{cor-rank7} states that $\cE_M$, or better its dual $\cE_M^\lor$, satisfies Theorem \ref{THM}, and it is therefore constructed as in Theorem \ref{mainthm} through the following commutative diagram 
$$
\xymatrix{
 & \vdots \ar[d] & \vdots \ar[d]  & \vdots \ar[d] &  \\
0 \ar[r] & \cO_{\PP^n}(-1)^{a} \ar[d]_M \ar[r] &  G^* \otimes \bigwedge^{c_1+1} F \otimes \cO_{\PP^n}(-1) \ar[d] \ar[r]^-{H_1} &  ((V^*)^{\lambda}) \otimes \cO_{\PP^n}(-1) \ar[r] \ar[d] & 0 \\
0 \ar[r] & \cO_{\PP^n}^{8} \ar[r] \ar[d] & \bigwedge^{c_1} F \otimes \cO_{\PP^n} \ar[r] \ar[d] & \cO_{\PP^n}^{\lambda} \ar[r] \ar[d] & 0\\
& \cO_{\PP^n}(c_1)\ar[d] \ar[r]^\simeq & \cO_{\PP^n}(c_1) \ar[d] & 0\\
& 0 & 0
}
$$
with $\lambda =\binom{n+c_1}{c_1}-8$, $F$, $G$ and $V$ vector spaces of dimension $n+c_1$, $c_1$ and $n+1$, respectively, and $H_1$ surjective.

If $c_1=2$ and $n=4$, then $a=5$ and $M$ cannot have rank 7.

If $c_1=2$ and $n=5$ or if $c_1=3$ and $n=3,4,5$, we have that
$$
c_1\binom{n+c_1}{c_1+1} - (n+1)\left(\binom{n+c_1}{c_1}-8\right)<0
$$
and therefore $H_1$ cannot be surjective.\\

\noindent\textbf{Extensions of Steiner and Drézet bundles.} We finally focus on the vector bundles defined by
$$
    0 \longrightarrow \cO_{\PP^n}(-1) \oplus \cO_{\PP^n}(-2) \longrightarrow \cO_{\PP^n}^9 \longrightarrow \cE_M \longrightarrow 0.
    $$
If $n=2$, then $\cE_M$ is forced to be the direct sum of the unique homogeneous rank 5 Drézet bundle with the twist of the tangent bundle $T_{\PP^2}(-1)$.
We  exclude the remaining case arguing as before. Here, the defining diagram is given by
$$
\xymatrix{
 & \vdots \ar[d] & \vdots \ar[d]  & \vdots \ar[d] &  \\
0 \ar[r] & \cO_{\PP^n}(-1)^{a} \ar[d]_M \ar[r] &  \left( \bigwedge^2 V^* \oplus G^* \otimes \bigwedge^{3} F\right) \otimes \cO_{\PP^n}(-1) \ar[d] \ar[r]^-{H_1} &  ((V^*)^{\lambda}) \otimes \cO_{\PP^n}(-1) \ar[r] \ar[d] & 0 \\
0 \ar[r] & \cO_{\PP^n}^{8} \ar[r] \ar[d] & \left(V^* \oplus \bigwedge^{2} F\right) \otimes \cO_{\PP^n} \ar[r] \ar[d] & \cO_{\PP^n}^{\lambda} \ar[r] \ar[d] & 0\\
& \cO_{\PP^n}(1)\oplus \cO_{\PP^n}(2) \ar[d] \ar[r]^\simeq & \cO_{\PP^n}(1)\oplus \cO_{\PP^n}(2) \ar[d] & 0\\
& 0 & 0
}
$$
with $\lambda = \binom{n+2}{2} + n - 7$, $F$, $G$ and $V$ vector spaces of dimension $n+2$, $2$ and $n+1$, respectively, and $H_1$ surjective.

If $n=3$, then $a=2$ and therefore $M$ cannot have rank 7.

If $n=4,5$, then
$$
\binom{n+1}{2} + 2\binom{n+2}{3} - (n+1)\left( \binom{n+2}{2} +n-7\right) < 0
$$
and therefore $H_1$ cannot be surjective.

\vskip 4mm

\subsection{Regular bundles}\label{sec-regularity}
The goal of this section is to prove the following result, that ensures the possibility to construct a  matrix of linear forms and constant rank  starting from the linear resolution of any regular bundle. For a sufficiently large twist, any vector bundle is regular. Hence the required hypotheses are not at all restrictive,

\begin{proposition}\label{prop-regmatrix}
  Let $s>0$ be an integer and $\cF$ a regular bundle on $\PP^n$ such that $H^0(\cF(-a))\neq 0$, with $a \geq 2$, $n\geq 2$. Then, for any $q \geq 2$ satisfying Condition (\ref{copies-surj}), such that $s$ generic elements in $H^0(\cF)^*$ define a maximal rank contraction, and $i=2,\ldots,q$, we have  a matrix $M_{i-1}$, of dimension
    $$
    \left(h^0(\cF \otimes \Omega^i(i)) -s \binom{n+1}{i-2}\right) \times \left(h^0(\cF \otimes \Omega^{i-1}(i-1)) -s \binom{n+1}{i-1}\right)
    $$
    of linear entries and constant rank $$
    \left(\sum_{k=0}^{i-2} (-1)^k \left(h^0(\cF \otimes \Omega^{i-k}(i-k)) -s \binom{n+1}{i-2-k}\right)\right) + (-1)^{i-1}\rk \cF.
    $$
\end{proposition}

Our goal is now to relate the linear resolution of a regular vector bundle  $\cF$ given in (\ref{res-regsheaf}), with the one given considering $\cO_{\PP^n}(1)$ as a regular bundle. Specifically, we will construct a commutative square of the following type
\begin{equation}\label{diag-regcontraction}
\xymatrix{
H^0(\cF \otimes \Omega^p(p)) \otimes \cO_{\PP^n}(-p) \ar[r] \ar[d] & H^0(\Omega^{p-1}(p)) \otimes \cO_{\PP^n}(-p) \ar[d] \\
H^0(\cF \otimes \Omega^{p-1}(p-1)) \otimes \cO_{\PP^n}(-p+1) \ar[r] & H^0(\Omega^{p-2}(p-1)) \otimes \cO_{\PP^n}(-p+1) 
}
\end{equation}
for any $p=1,\ldots,n$. The vertical maps are defined by the Beilinson spectral sequence and, in particular, by the choice of the ``identity element'' $\sum_{i=0}^n v_i \otimes v^*_i$. The horizontal maps are defined, thanks to a simple diagram chase, by the contraction induced by a chosen element $\omega \in H^0(\cF)^*$. The required maps are related to the contraction mentioned above in light of the following commutative triangle
$$
\xymatrix{
H^0(\cF \otimes \Omega^p(p)) \ar[rd] \ar[r] & H^0(\cF) \otimes H^0(\Omega^p(p))\simeq H^0(\cF) \otimes \bigwedge^{p-1}V^* \ar[d]^{- \wedge \omega}\\
& H^0(\Omega^{p-1}(p)) \simeq \bigwedge^{p-1}V^*
}
$$
To ensure the hypotheses of Theorem \ref{mainthm}, we will need the following technical result.

\begin{lemma}\label{lemma-ggbundles}
    Let $\cF$ and $\cG$ be two regular vector bundles, jointly with an injective map $\cG \hookrightarrow \cF$. Then, for $p=1,\ldots,n$, we have the following commutative diagram:
    $$
    \xymatrix{
    H^0(\cG \otimes \Omega^p(p)) \otimes \cO_{\PP^n}(-p) \ar@{^{(}->}[d] \ar[r] & H^0(\cG \otimes \Omega^{p-1}(p-1)) \otimes \cO_{\PP^n}(-p+1) \ar@{^{(}->}[d]\\
    H^0(\cF \otimes \Omega^p(p)) \otimes \cO_{\PP^n}(-p) \ar[r] & H^0(\cF \otimes \Omega^{p-1}(p-1)) \otimes \cO_{\PP^n}(-p+1)
    }
    $$
    where the vertical maps are given by the inclusion of global sections and the horizontal maps are induced by the regularity resolution.
\end{lemma}

\begin{proof}
    We take the commutativity diagram
    $$
    \xymatrix{
    0 \ar[r] & \mathcal{K}_1 \ar[d]_{i_1} \ar[r] & H^0(\cG) \otimes \cO_{\PP^n} \ar[d]^{\alpha_0} \ar[r]^-{ev_G} & \cG \ar@{^{(}->}[d]^i \ar[r] & 0\\
    0 \ar[r] & \mathcal{H}_1 \ar[r] & H^0(\cF) \otimes \cO_{\PP^n} \ar[r]^-{ev_F} & \cF  \ar[r] & 0\\
    }
    $$
    obtained from the injective map $i$. Indeed, $ev_G \circ i$ lifts to the morphism $\alpha_0$ since $\Ext^1(\cO_{\PP^n},\mathcal{H}_1)= 0$. Furthermore, the map $\alpha_0$ is injective. If it was not, its kernel would be given by copies of $\cO_{\PP^n}$ that inject in $\mathcal{K}_1$, making them a direct summand of this last vector bundle, which is a contradiction. Thierefore, $i_1$ is injective as well.

    Iterating the previous construction, we have a commutative diagram
    $$
    \xymatrix{
    0 \ar[r] & \mathcal{K}_{p+1} \ar[d]_{i_{p+1}} \ar[r] & H^0(\cG \otimes \Omega^p(p)) \otimes \cO_{\PP^n}(-p) \ar[d]^{\alpha_p} \ar[r] & \mathcal{K}_{p} \ar@{^{(}->}[d]^{i_p} \ar[r] & 0\\
    0 \ar[r] & \mathcal{H}_{p+1} \ar[r] & H^0(\cF \otimes \Omega^p(p)) \otimes \cO_{\PP^n}(-p) \ar[r] & \mathcal{H}_p  \ar[r] & 0\\
    }
    $$
    given by the regularity of $\mathcal{H}_{p+1}(p+1)$. As before, the map $\alpha_p$ is injective, and consequently $i_{p+1}$ as well.
 This proves our result.
\end{proof}

Consider now a regular bundle $\cF$ such that $H^0(\cF(-a)) \neq 0$ with $a \geq 2$. Due to Lemma \ref{lemma-ggbundles} and Diagram (\ref{diag-regcontraction}), we have the commutative diagram
$$
\xymatrix{
H^0(\cF \otimes \Omega^p(p)) \otimes \cO_{\PP^n}(-p) \ar[d]_\simeq  \ar[rd] \\ \left(H^0(\Omega^p(p+a)) \otimes \cO_{\PP^n}(-p)\right) \oplus \coker \alpha_p \ar[r] \ar[d] & H^0(\Omega^{p-1}(p)) \otimes \cO_{\PP^n}(-p) \ar[d] \\
\left(H^0(\Omega^{p-1}(p-1+a)) \otimes \cO_{\PP^n}(-p+1)\right) \oplus \coker \alpha_{p-1} \ar[r] & H^0(\Omega^{p-2}(p-1)) \otimes \cO_{\PP^n}(-p+1)\\
H^0(\cF \otimes \Omega^{p-1}(p-1)) \otimes \cO_{\PP^n}(-p+1) \ar[ur] \ar[u]^\simeq & 
}
$$
for any $p=1,\ldots,n$. The surjectivity of the horizontal maps is ensured if we fix, to define the required contraction, an element $\omega \in H^0(\cF)^*$ whose paired element, in the original vector space, belongs to $H^0(\cO_{\PP^n}(a))$ and defines a diagram of type
$$
\xymatrix{
H^0(\Omega^p(p+a)) \otimes \cO_{\PP^n}(-p) \ar@{->>}[r] \ar[d] & H^0(\Omega^{p-1}(p)) \otimes \cO_{\PP^n}(-p) \ar[d] \\
H^0(\Omega^{p-1}(p-1+a)) \otimes \cO_{\PP^n}(-p+1) \ar@{->>}[r] & H^0(\Omega^{p-2}(p-1)) \otimes \cO_{\PP^n}(-p+1).\\
}
$$
Indeed, the previous diagram is the one given by the maximal rank contraction relating the Eagon-Northcott and Koszul complexes, as detailed in the previous part of this section. Its surjectivity (in the horizontal maps) obviously implies the same ``horizontal'' surjectivity in Diagram (\ref{diag-regcontraction}), and a simple dimension computation proves therefore Proposition \ref{prop-regmatrix}.


\begin{example}
We consider the Horrocks-Mumford bundle $\cH$, defined in \cite{HM}. Recall that it is an indecomposable rank 2 bundle on $\PP^4$, which can be defined as the cohomology of the following monad
$$
0 \longrightarrow \cO_{\PP^4}(-1)^{5} \longrightarrow \Omega^2_{\PP^4}(2)^{2} \longrightarrow\cO_{\PP^4}^{5} \longrightarrow 0.
$$
Its resolution can be computed as described in \cite{D} or can be obtained using Macaulay2, with the aid of the function {\tt beilinson}, as explained in \cite{DE}. Furthermore, we know that it is 5-regular and the linear resolution of $\cF := \cH(5)$, that can be also computed with Macaulay2, is given by
{\small
\begin{equation}\label{res-HM}
0 \longrightarrow \cO_{\PP^4}(-4)^{35} \longrightarrow \cO_{\PP^4}(-3)^{173} \longrightarrow \cO_{\PP^4}(-2)^{330} \longrightarrow \cO_{\PP^4}(-1)^{290} \longrightarrow \cO_{\PP^4}^{100} \longrightarrow \cF \longrightarrow 0
\end{equation}
}
We consider the generic (and therefore surjective) map 
$$
k^{100} \simeq H^0(\cF) \longrightarrow k^{33} \simeq H^0(\cO_{\PP_4}^{33}).
$$
This map, through a diagram chase, can be lifted to a map of chain complexes, between the resolution of $\cF$ and the direct sum of 33 copies of the Koszul complex. In particular we obtain the following commutative diagram
\begin{equation}\label{HM-diag}
\xymatrix{
& & H^0(\cF \otimes \Omega^2(2)) \otimes \cO_{\PP^4}(-2) \ar[r]^\simeq \ar[d] & H^0(\Omega^{1}(2)) \otimes \cO_{\PP^4}(-2)^{33} \ar[d]  \\
0 \ar[r] & W_1 \otimes \cO_{\PP^4}(-1) \ar[r] \ar[d]_M & H^0(\cF \otimes \Omega^{1}(1)) \otimes \cO_{\PP^4}(-1) \ar[d] \ar[r] & H^0(\cO_{\PP^4}(1)) \otimes \cO_{\PP^4}(-1)^{33} \ar[d] \ar[r] & 0\\
0 \ar[r] & W_0 \otimes \cO_{\PP^4} \ar[r] & H^0(\cF) \otimes \cO_{\PP^4} \ar[r] &  \cO_{\PP^4}^{33} \ar[r] & 0
}
\end{equation}
whose dimensions are given by 
$$
\xymatrix{
& 330 \ar[r] \ar[d] & 330 \ar[d]\\
 125\ar[d]_M \ar[r] & 290\ar[d] \ar[r] & 165 \ar[d]\\
 67 \ar[r] & 100 \ar[r] & 33 
}
$$
Thanks to the commutativity of the previous diagram, we determine the indecomposable  $(67 \times 125)$-matrix $M$, with linear entries and  has  constant  rank 65.
This construction can be tested by running the Macaulay2 code provided in Appendix B.
\end{example}


\section{The moduli space}\label{sec-moduli}

In this last section, we study families of matrices $M$ of constant rank in terms of the corresponding vector bundles $\cE_M$ and their moduli spaces.
Denote by $\Spl(r;c_1,\cdots ,c_n)$ the moduli space of simple sheaves on $\PP^n$ of rank $r$ and given Chern classes $c_i$. 

 Using  recent results on generalized syzygy bundles, we will  prove that, under some mild hypothesis, the rank $r$ vector bundles $\cE_M$  are simple, unobstructed (i.e., the corresponding points in the moduli space are smooth) and they are parametrized by an open dense subset of an irreducible component of   $\Spl(r;c_1(\cE_M),\cdots ,c_n(\cE_M))$.

It is important to point out that we have to consider the moduli space of simple sheaves  $\Spl(r;c_1(\cE_M),\cdots ,c_n(\cE_M))$, instead of the better  understood moduli space of (semi)stable sheaves $M^{ss}(r;c_1(\cE_M),\cdots ,c_n(\cE_M)))$, because in general $\cE_M$ is neither stable nor semistable while quite often it is simple.

Let us start by collecting, for
the reader’s convenience, some basic information about moduli spaces of simple sheaves. For more information the reader can look at \cite{FM}.

Let $\cF$ be a globally generated vector bundle of rank $r$ on $\PP^n$. For any  general subspace $W\subset H^0(\cF)$ with $w=\dim W\ge n+r$, we define a vector bundle $\cS$ on $\PP^n$ as the kernel  of the map $eval _W:W\otimes \cO_{\PP^n}\lra \cF$. So, $\cS$   fits inside a short exact sequence:
\begin{equation}
    \label{exact1}
e: \quad 0\to  \cS\to W\otimes \mathcal O_{\PP^n} \to  \cF\to 0, \text{ and }
\end{equation}
$$\begin{array}{rcl} \rk(\cS)& = & \dim W- \rk(\cF)= w-r, \\
c_1(\cS) & = & -c_1(\cF) \\
c_t(\cS)\cdot c_t(\cF)& = & 1.
\end{array}
$$
We will call  $\cS$ the \textit{generalized syzygy bundle} associated to the couple $(\cF,W)$. Observe that the dual of a generalized syzygy bundle associated to a non complete linear system $W\subset H^0(\cO_{\PP^n}(d))$ is nothing but a Drézet bundle and, if $W=H^0(\PP^n,\cO _{\PP^n}(d))$, we recover the classical definition of syzygy bundle.

 We will denote by $$U:=U(r;c_1,\cdots ,c_n)\subset  \Spl(r; c_1,\cdots ,c_n)$$ the open locus parametrizing  globally generated rank $r$ simple vector bundles $\cF$ on $\PP^n$ such that $H^1(\cF)=H^1(\cF^\lor)=H^2(\cF^\lor)=0$.
 
 \begin{remark} \label{cte}
 By \cite[Proposition 3.2]{FM}, we know that the set-theoretic function associating to each point $[\cF] $ of this open locus $U$ the dimension of $H^0(\cF)$ is locally constant. Hence it is constant on every connected component of $U$.    
 \end{remark}

 From now on, we assume that the above open subset $U$ is connected; otherwise, we will restrict our discussion to each connected component.
 Define
$$\cG _U:=\{(\cF,W) \mid \cF\in U, W\subset H^0(\cF), \dim (W)=w\}.$$
The natural projection $\pi: \cG_U\to U$ is a Grassmann bundle and  $\cG_U$ is a smooth algebraic space of dimension $$\dim \cG_U=\dim \Spl(r;c_1,\cdots ,c_n)+ w(z-w)$$ where $z:=\dim H^0(\PP^n,\cF)$ (see Remark \ref{cte}). 
Moreover, consider 
$$\cG^0 _U:=\{(\cF,W)\in \cG _U  \mid eval_W \text{ is surjective}\},$$
which is open in $\cG _U$. We still denote by $\pi $ the induced natural projection $\pi : \cG^0_U \to U$.

 It is possible to check that, for any pair $(\cF,W)\in \cG^0_U$, the associated generalized syzygy bundle is simple and, using stack theory, we can define  a morphism (see \cite[Definition 3.7]{FM} for details)
 $$\alpha:\cG^0_U\lra \Spl(\rk(\cS);c_1(\cS),\cdots ,c_n(\cS))$$ which extends the set-theoretic map $(\cF,W)\mapsto \cS$.\\
It is known that:
\begin{itemize}
    \item[(i)]
 The morphism $\alpha$ is injective and, if $n\ge 3$, it is an open embedding \cite[Theorem  1.2, Proposition 4.1]{FM};
\item[(ii)] We have $$T_{[(\cF,W)]}\cG^0_U=H^0(\cS^\lor\otimes V)/W\otimes W^*,$$ $$T_{[\cS]}\Spl(\rk(\cS);c_1(\cS),\cdots ,c_n(\cS))=\Ext^1(\cS,\cS)=H^1(\cS^\lor\otimes \cS)$$ and the differential map  $$d \alpha :T_{[(\cF,W)]}\cG^0_U\lra T_{[\cS]}\Spl(\rk(\cS);c_1(\cS),\cdots ,c_n(\cS))$$ is injective \cite[Proposition 4.2]{FM}.
 
\end{itemize}

We are now ready to state the main result of this section.

\begin{theorem}\label{moduli}  Let $M$ be a $(b\times a)$-matrix with linear entries and constant rank $r$, which fits in
$$
\xymatrix{
0 \ar[r] & \cK_M \ar[r] &  \cO_{\PP^n}^{a}\ar[rr]^\psi\ar[rd] &
& \cO_{\PP^n}(1)^{b}\ \ar[r] & \cC_M \ar[r] & 0.  \\
& & & \cE_M\ar[ru]& && \\
}
$$
Assume that $\cC_M$ is  simple and $H^1(\cC_M(-1))=H^1(\cC_M^\lor(1))=H^2(\cC_M^\lor(1))=0$. Then, $\cE_M$ is a simple vector bundle  parametrized by a smooth  open subset ${ \bf S}$ of $\Spl(r;c_1(\cE_M(-1)),\cdots ,c_n(\cE_M(-1)))$.   Moreover, $$\dim  {\bf S} =\dim _{[\cC_M(-1)]}\Spl(b-r;c_1(\cC_M(-1),\cdots ,c_n(\cC_M(-1)))+ b(h^0(\cC_M(-1))-b).$$   
\end{theorem}
\begin{proof}    

Consider the open set $U$ and the pairs in $\cG^0_U$ as defined before. By \cite{FM}, the syzygy bundle associated to any pair $(\cF,W)\in \cG^0_U$ is simple. In particular, taking $\cF = \cC_M(-1)$, we obtain that $\cE_M(-1)$ (and therefore $\cE_M$) is simple. 

The result immediately follows from the fact that $\alpha$ is an open embedding.
\end{proof}

\begin{remark}
Arguing as in \cite{BMR} and \cite{E}, we deduce that  the number $n(r,c_1,c_2,...)$ of irreducible components of the moduli space $\Spl(r;c_1,\cdots ,c_n)$ increases when $c_2$ increases. Therefore, the simple vector bundles associated to a $b\times a$ matrix of constant rank $r$ only fill an irreducible component. It is worthwhile to point out that this component has nice properties: it is generically smooth and we know its dimension. 
\end{remark}

\begin{example}
(1) Since $\Omega _{\PP^5}^1(2)$ is  rigid, simple,  generated by global sections and $H^1(\Omega _{\PP^5}^1(2))= H^1(T_{\PP^5}(-2))=H^2(T_{\PP^5}(-2))=0$, as an application of Theorem \ref{moduli} we get that the indecomposable $14\times 14$ matrices, with entries linear forms in 6 variables and constant rank 9, are parametrized by a smooth variety of dimension 14.

\vskip 2mm
(2)  Since $\cO _{\PP^3}(3)$ is  rigid, simple,  generated by global sections and $H^1(\cO _{\PP^3}(3))=H^1(\cO _{\PP^3}(-3))=H^2(\cO _{\PP^3}(-3))=0$, by Theorem \ref{moduli} we get that the indecomposable  $25\times 15$ matrices, with entries linear forms in 4 variables and constant rank 14, are parametrized by an irreducible open subset of dimension 51 of $\Spl(14;-3,9,-27)$.
\end{example}


\section{Final comments and questions}\label{sec-openpb}

In this section, we list some questions and open problems stemming from this work.

\vskip 2mm

(1) {\bf Homogeneous matrices of constant rank.} We would like to construct $b\times a$ matrices $M_d$ with entries homogeneous  forms of degree $d$ in $R$ and constant rank $r$. Again this problem can be translated into a problem of vector bundles, since any such matrix defines a
morphism of sheaves $\psi: \cO_{\PP^n}^{a}\longrightarrow \cO_{\PP^n}(d)^{b}$. If the rank is constant, equal to $r$, the image of this
morphism is a vector bundle $\cE$ on $\PP^n$ of rank $r$.

We first observe that any matrix $M_1$, with linear entries and constant rank $r$, gives rise to a matrix $M_d$ of the same size, constant rank $r$ and entries homogeneous forms of degree $d$; it is enough to change $x_i$ by $x_i^d$, for $i=0,\cdots, n$.  Nevertheless, not all ``degree $d$ matrices'' of type $M_d$ come from ``linear matrices'' of type $M_1$ of the same size. Indeed, the $5\times 5$ matrices $M_1$ with linear entries in $R=k[x,y,z]$ and constant rank $4$ are parametrized by an open subset of $\PP^5$   while the $5\times 5$ matrices $M_2$ with quadratic  entries in $R=k[x,y,z]$ and constant rank $4$ are parametrized by an open subset of $\PP^{35}$.   

So, we are led to pose the following problem:
\begin{problem}
    (i) Given $n$ and $d$, to construct constant rank matrices $M_d$ with entries homogeneous forms of degree $d$ in $n+1$ variables.

    (ii) Given $n$, $d$, $a$ and $b$, to determine the minimum $r$ (if it exists) such that there are $(b\times a)$-matrices with entries homogeneous forms of degree $d$ in $n+1$ variables and constant rank $r$.
\end{problem}
\vskip 2mm

(2) {\bf Stability of vector bundles associated to matrices with constant rank.} In this work we often use the following recipe to construct matrices with linear entries and constant rank.
We start with a vector bundle $\cC_M$ on $\PP^n$ with a linear resolution
$$
0 \longrightarrow  \cO_{\PP^n}(-n)^{a_n} \longrightarrow  \cO_{\PP^n}(-n+1)^{a_{n-1}} \longrightarrow  
\cdots  
$$
$$\cdots \longrightarrow  \cO_{\PP^n}(-1)^{a_1} \longrightarrow H^0(\cC_M(-1))\otimes \cO_{\PP^n} \longrightarrow \cC_M(-1) \longrightarrow 0.
$$
We then determine the minimal dimension vector subspace $ W \subset H^0(\cC_M(-1))$ that generates $\cC_M(-1)$ at every point and ensures a linear resolution of $W\otimes \cO_{\PP^n}\longrightarrow \cC_M(-1)\longrightarrow 0$.
Finally, we observe that if $\cC_M$ is simple then $\cE_M(-1):=ker(W\otimes \cO_{\PP^n}\rightarrow \cC_M(-1)) $ is also simple.\\
We are led to pose the following question:

\begin{question}
    Assume $\cC_M$ to be stable. Is there a choice for $W\subset H^0(\cC_M(-1))$ to ensure the stability of $\cE_M$?
\end{question}

\vskip 2mm

(3) Let $(X,L)$ be a polarized variety. As a generalization of the method developed in the previous sections we propose to address the same problem in a more general setup and construct matrices with entries in $H^0(\cO_X(L))$ of constant rank.

\vskip 2mm

(4) {\bf Maximal rank contraction.} In this paper we quite often deal with maximal rank contractions. More precisely, we take a vector space $V$ of dimension $n$, we consider the exterior algebra $\bigwedge V$, we pick a  $q$-form $\omega\in \bigwedge ^q V^*$ and we assume that it induces a maximal rank morphism $\omega\wedge -:\bigwedge ^pV\longrightarrow \bigwedge ^{p-q}V$ (see Definition \ref{def-contraction}). Nevertheless, we want to point out that in all the cases we  have considered it was enough to pick up a sufficiently general $q$-form to guarantee the maximal rank property. Therefore, based in Examples \ref{ex-maxrank} and \ref{ex-notmaxrank} we are led to pose the following problem.

\begin{question}\label{maxrk}
    Let $V$ be a vector space of dimension $n$.
    For which integers $p>q$ there exists a general $q$-form such that the induced contraction morphism 
    $$
    \omega\wedge -:\bigwedge ^pV\longrightarrow \bigwedge ^{p-q}V
    $$ 
    has maximal rank?
\end{question}

\begin{remark} Although the above question seems a simple problem of linear algebra, it has proven to be extremely elusive and the answer is not always affirmative.
Consider for instance  a vector space $V$ of dimension 9; hence $\dim \bigwedge ^3V = dim \bigwedge ^6V= 84$. We would expect that, for a general 3-form $\omega$, the induced contraction morphism 

     $$
    \omega\wedge -:\bigwedge ^3V\longrightarrow \bigwedge ^{6}V
    $$
    is an isomorphism. However, for a general 3-form $\omega $ we have  $$
    \dim Ker (\omega\wedge -:\bigwedge ^3V\longrightarrow \bigwedge ^{6}V)=4.
    $$

    As recalled in the introduction, Question 6.3 is equivalent to the long-standing problem of determining the Hilbert function of the exterior algebra modulo a principal ideal generated by a generic form of degree $q$. For more information about this problem, the reader can look at \cite{LN} and \cite{MS}.
\end{remark}

\appendix
\section{The explicit construction of a ($15\times 25$)-matrix with linear entries and constant rank 14}\label{example-appendix}
In this part, we want to explicitly describe how to construct the matrix provided in Example \ref{ex-14}(2), namely an indecomposable ($15\times 25$)-matrix $M$ of constant rank 14 and entries linear forms in $k[x,y,z,t]$.
Recall that, in order to apply Theorem \ref{mainthm}, we have to provide the following commutative diagram
$$
\xymatrix{
0 \ar[r] & W_2 \otimes \cO_{\PP^3}(-2) \ar[r] \ar[d]_{\mathfrak{d}_2} & S_2 G^* \otimes \bigwedge^5 F \otimes \cO_{\PP^3}(-2) \ar[d]^{D_{5,2}} \ar[r]^{H_2} & \left(\bigwedge^2 V^*\right)^{5} \otimes \cO_{\PP^3}(-2) \ar[r] \ar[d] & 0 \\
0 \ar[r] & W_1 \otimes \cO_{\PP^3}(-1) \ar[d] \ar[r]^-{S_1} &  G^* \otimes \bigwedge^4 F \otimes \cO_{\PP^3}(-1) \ar[d]^{D_{4,1}} \ar[r]^{H_1} &  ((V^*)^{5}) \otimes \cO_{\PP^3}(-1) \ar[r] \ar[d] & 0 \\
0 \ar[r] & W_0 \otimes \cO_{\PP^3} \ar[r]^-{S_0} & \bigwedge^3 F \otimes \cO_{\PP^3} \ar[r]^{H_0} & \cO_{\PP^3}^{5}   & 
}
$$
where $H_1$ and $H_2$ have to be surjective.\\

Consider an element $\omega \in \bigwedge^{3}F^*$, which we express, in terms of the basis, as:
$$
\omega=\sum_{0\leq i < j < k \leq 5} \alpha_{(i,j,k)} f^*_i \wedge f^*_j \wedge f^*_k.
$$
The matrix $H_{(1,4,1)}^\omega$, representing the map
$$
G^* \otimes \bigwedge^{4}F  \otimes \cO_{\PP^3}(-1) \rightarrow  V^* \otimes \cO_{\PP^3}(-1),
$$
is constructed as follows. Consider the contraction $\bigwedge^{4}F \rightarrow F$, defined by $\omega$, represented by (the transposed of) the matrix
{\tiny
$$
P_{(\{4,1\})}^T = 
\left(
\begin{array}{cccccccccccccccccccccccccccccccc}
\alpha_{(1,2,3)} & -\alpha_{(0,2,3)} & \alpha_{(0,1,3)} & -\alpha_{(0,1,2)} & 0 & 0 \vspace{1mm}\\
\alpha_{(1,2,4)} & -\alpha_{(0,2,4)} & \alpha_{(0,1,4)} & 0 & -\alpha_{(0,1,2)} & 0\vspace{1mm}\\
\alpha_{(1,2,5)} & -\alpha_{(0,2,5)} & \alpha_{(0,1,5)} & 0 & 0 & -\alpha_{(0,1,2)}\vspace{1mm}\\
\alpha_{(1,3,4)} & -\alpha_{(0,3,4)} & 0 & \alpha_{(0,1,4)} & -\alpha_{(0,1,3)} & 0\vspace{1mm}\\
\alpha_{(1,3,5)} & -\alpha_{(0,3,5)} & 0 & \alpha_{(0,1,4)} & 0 & -\alpha_{(0,1,3)}\vspace{1mm}\\
\alpha_{(1,4,5)} & -\alpha_{(0,4,5)} & 0 & 0 & \alpha_{(0,1,5)} & -\alpha_{(0,1,4)} \vspace{1mm}\\
\alpha_{(2,3,4)} & 0 & -\alpha_{(0,3,4)} & \alpha_{(0,2,4)} & -\alpha_{(0,2,3)} & 0 \vspace{1mm}\\
\alpha_{(2,3,5)} & 0 & -\alpha_{(0,3,5)} & \alpha_{(0,2,5)} & 0 & -\alpha_{(0,2,3)} \vspace{1mm}\\
\alpha_{(2,4,5)} & 0 & -\alpha_{(0,4,5)} & 0 & \alpha_{(0,2,5)} & -\alpha_{(0,2,4)} \vspace{1mm}\\
\alpha_{(3,4,5)} & 0 & 0 & -\alpha_{(0,4,5)} & \alpha_{(0,3,5)} & -\alpha_{(0,3,4)} \vspace{1mm}\\
0 & \alpha_{(2,3,4)} & -\alpha_{(1,3,4)} & \alpha_{(1,2,4)} & -\alpha_{(1,2,3)} & 0 \vspace{1mm}\\
0 & \alpha_{(2,3,5)} &  -\alpha_{(1,3,5)} & \alpha_{(1,2,5)} & 0 & -\alpha_{(1,2,3)}\vspace{1mm}\\
0 & \alpha_{(2,4,5)} &  -\alpha_{(1,4,5)} & 0 & \alpha_{(1,2,5)} & -\alpha_{(1,2,4)}\vspace{1mm}\\
0 & \alpha_{(3,4,5)} & 0 & -\alpha_{(1,4,5)} & \alpha_{(1,3,5)} & -\alpha_{(1,3,4)} \vspace{1mm}\\
0 & 0 & \alpha_{(3,4,5)} & -\alpha_{(2,4,5)} & \alpha_{(2,3,5)} & -\alpha_{(2,3,4)} 
\end{array}
\right)
$$
}
from which we can induce the matrix representing $G^* \otimes \bigwedge^{4}F \rightarrow G^* \otimes F$ by considering
$$
P_{(\{4,1\},\{1,1\})} =
\left(
\begin{array}{c|c|c}
    P_{(\{4,1\})} & O & O  \vspace{1mm}\\
    \hline
    O & P_{(\{4,1\})} & O \vspace{1mm}\\
    \hline
    O & O & P_{(\{4,1\})}
\end{array}
\right)
$$
where $O$ denotes here a $(6 \times 15)$-matrix with all entries zero. Finally, we multiply the previous matrix with the one representing the map
$$
G^* \otimes F \longrightarrow V
$$
that is
{\tiny
$$
H_{1,1} = 
\left(
\begin{array}{ccccc}
    1 & 0 & 0 & 0\\ 
    0 & 1 & 0 & 0\\
    0 & 0 & 1 & 0\\
    0 & 0 & 0 & 1\\
    0 & 0 & 0 & 0\\
    0 & 0 & 0 & 0\\
    0 & 0 & 0 & 0\\
    1 & 0 & 0 & 0\\
    0 & 1 & 0 & 0\\
    0 & 0 & 1 & 0\\
    0 & 0 & 0 & 1\\
    0 & 0 & 0 & 0\\
    0 & 0 & 0 & 0\\
    0 & 0 & 0 & 0\\
    1 & 0 & 0 & 0\\
    0 & 1 & 0 & 0\\
    0 & 0 & 1 & 0\\
    0 & 0 & 0 & 1
\end{array}
\right)
$$
}
Choosing randomly the coefficients $\alpha_{(i,j,k)}$ of $\omega$ to obtain five generic elements in $\bigwedge^{3} F^*$, we get
$$
H_1 = 
\left(
\begin{array}{ccc}
H_{(1,4,1)}^{\omega_1} \vspace{1mm}\\
\hline
H_{(1,4,1)}^{\omega_2}\vspace{1mm}\\
\hline
H_{(1,4,1)}^{\omega_3}\vspace{1mm}\\
\hline
H_{(1,4,1)}^{\omega_4}\vspace{1mm}\\
\hline
H_{(1,4,1)}^{\omega_5}
\end{array}
\right).
$$
Again using Macaulay2, we can check that we have obtained a $(20 \times 45)-$matrix of maximal rank. We obtain $S_1$ directly as the syzygy matrix of $H_1$. Analogously, the matrix $H_0$ defining
$$
\bigwedge^3 F \otimes \cO_{\PP^3} \longrightarrow \left(\cO_{\PP^3}\right)^{5} 
$$
is obtained gluing the five matrices, each one obtained by the choice of a 3-form $\omega \in \bigwedge^3 F^*$, constructed as the $(1\times 20)$- matrix
$$
H^\omega_{(0,3,0)} = 
\left\{\alpha_{(i,j,k)}\right\}_{0\leq i < j < k \leq 5},
$$
with the entries ordered by the three-index. As before, let us denote by $S_0$ the syzygy matrix of $H_0$.
Following the definition described in (\ref{derivations-strand}), we build directly the matrix $D_{4,1}$ representing the map
$$
\delta_{4,1} : G^* \otimes \bigwedge^4 F \otimes \cO_{\PP^3}(-1) \longrightarrow \bigwedge^3 F \otimes \cO_{\PP^3},
$$ 
once chosen the canonical basis
$$
\left\{g_i \otimes (f_{j_1} \wedge f_{j_2} \wedge f_{j_3} \wedge f_{j_4})\right\}_{i=0,1,2 \atop 0 \leq j_1 < j_2 < j_3 < j_4 \leq 5} \mbox{ for } G^* \otimes \bigwedge^4 F
$$
and 
$$
\left\{f_{j_1} \wedge f_{j_2} \wedge f_{j_3})\right\}_{0 \leq j_1 < j_2 < j_3 \leq 5} \mbox{ for } \bigwedge^3 F.
$$

Notice that, with the matrix we have just defined, the product $H_0 \cdot Q \cdot W$ is equal to a zero matrix, giving evidence to the fact that the maps considered define a commutative diagram.
Before finding our \textit{final matrix}, let us show how to check that $H_2$ is also a surjective map. As before, we start by representing the contraction $\bigwedge^5 F \longrightarrow \bigwedge^2 F$, given by the element $\omega$, which gives the matrix
{\tiny
$$
P_{(\{5,2\})}=
\left(
\begin{array}{cccccccccccccccccccccccccccc}
     \alpha_{(2,3,4)} & \alpha_{(2,3,5)} & \alpha_{(2,4,5)} & \alpha_{(3,4,5)} & 0 & 0\\
     -\alpha_{(1,3,4)} & -\alpha_{(1,3,5)} & -\alpha_{(1,4,5)} & 0 & \alpha_{(3,4,5)} & 0\\
     \alpha_{(1,2,4)} & \alpha_{(1,2,5)} & 0  & -\alpha_{(1,4,5)} & -\alpha_{(2,4,5)} & 0\\
     -\alpha_{(1,2,3)} & 0 & \alpha_{(1,2,5)} & \alpha_{(1,3,5)} & \alpha_{(2,3,5)} & 0\\
     0 & -\alpha_{(1,2,3)} & -\alpha_{(1,2,4)} & -\alpha_{(1,3,4)} & -\alpha_{(2,3,4)} & 0\\
     \alpha_{(0,3,4)} & \alpha_{(0,3,5)} & \alpha_{(0,4,5)} & 0 & 0 & \alpha_{(3,4,5)}\\
    -\alpha_{(0,2,4)} & -\alpha_{(0,2,5)} & 0 & \alpha_{(0,4,5)} & 0 & -\alpha_{(2,4,5)}\\
    \alpha_{(0,2,3)} & 0 & -\alpha_{(0,2,5)} & -\alpha_{(0,3,5)} & 0 & \alpha_{(2,3,5)}\\
    0 & \alpha_{(0,2,3)} & \alpha_{(0,2,4)} & \alpha_{(0,3,4)} & 0 & -\alpha_{(2,3,4)}\\
    -\alpha_{(0,1,4)} & -\alpha_{(0,1,5)} & 0 & 0 & \alpha_{(0,4,5)} & \alpha_{(1,4,5)}\\
    \alpha_{(0,1,3)} & 0 & -\alpha_{(0,1,5)} & 0 & -\alpha_{(0,3,5)} & -\alpha_{(1,3,5)}\\
    0 & \alpha_{(0,1,3)} & \alpha_{(0,1,4)} & 0 & \alpha_{(0,3,4)} & \alpha_{(1,3,4)}\\
    -\alpha_{(0,1,2)} & 0 & 0 & -\alpha_{(0,1,5)} & -\alpha_{(0,2,5)} & -\alpha_{(1,2,5)}\\
    0 & -\alpha_{(0,1,2)} & 0 & \alpha_{(0,1,4)} & \alpha_{(0,2,4)} & \alpha_{(1,2,4)}\\
    0 & 0 & -\alpha_{(0,1,2)} & -\alpha_{(0,1,3)} & -\alpha_{(0,2,3)} & -\alpha_{(1,2,3)}
\end{array}
\right)
$$
}
from which we can induce the matrix representing $S_2 G^* \otimes \bigwedge^{5}F \rightarrow S_2 G^* \otimes \bigwedge^2 F$ by considering
$$
P_{(\{5,2\},\{2,2\})} =
\left(
\begin{array}{c|c|c|c|c|c}
    P_{(\{5,2\}} & O & O & O & O & O  \vspace{1mm}\\
    \hline
    O & P_{(\{5,2\}} & O & O & O & O\vspace{1mm}\\
    \hline
    O & O & P_{(\{5,2\}} & O & O & O\vspace{1mm}\\
    \hline
    O & O & O & P_{(\{5,2\}} & O & O\vspace{1mm}\\
    \hline
    O & O & O & O & P_{(\{5,2\}} & O\vspace{1mm}\\
    \hline
    O & O & O &O & O & P_{(\{5,2\}}
\end{array}
\right)
$$
Once again, we multiply the previous matrix by the (transposed of the) following one, associated to the map $S_2 G^* \otimes \bigwedge^2 F \longrightarrow \bigwedge^2 V^* $,
{\tiny
$$
H_{2,2}^T =
\left(
\begin{array}{cccccccccccccccccccccccccc}
    2 & 0 & 0 & 0 & 0 & 0 & 0 & 0 & 0 & 0 & 0 & 0 & 0 & 0 & 0\\
    0 & 2 & 0 & 0 & 0 & 0 & 0 & 0 & 0 & 0 & 0 & 0 & 0 & 0 & 0\\
    0 & 0 & 2 & 0 & 0 & 0 & 0 & 0 & 0 & 0 & 0 & 0 & 0 & 0 & 0\\
    0 & 0 & 0 & 0 & 0 & 2 & 0 & 0 & 0 & 0 & 0 & 0 & 0 & 0 & 0\\
    0 & 0 & 0 & 0 & 0 & 0 & 2 & 0 & 0 & 0 & 0 & 0 & 0 & 0 & 0\\
    0 & 0 & 0 & 0 & 0 & 0 & 0 & 0 & 0 & 2 & 0 & 0 & 0 & 0 & 0\\
    0 & 1 & 0 & 0 & 0 & 0 & 0 & 0 & 0 & 0 & 0 & 0 & 0 & 0 & 0\\
    0 & 0 & 1 & 0 & 0 & 1 & 0 & 0 & 0 & 0 & 0 & 0 & 0 & 0 & 0\\
    0 & 0 & 0 & 1 & 0 & 0 & 1 & 0 & 0 & 0 & 0 & 0 & 0 & 0 & 0\\
    0 & 0 & 0 & 0 & 0 & 0 & 1 & 0 & 0 & 0 & 0 & 0 & 0 & 0 & 0\\
    0 & 0 & 0 & 0 & 0 & 0 & 0 & 1 & 0 & 1 & 0 & 0 & 0 & 0 & 0\\
    0 & 0 & 0 & 0 & 0 & 0 & 0 & 0 & 0 & 0 & 1 & 0 & 0 & 0 & 0\\
    0 & 0 & 1 & 0 & 0 & -1 & 0 & 0 & 0 & 0 & 0 & 0 & 0 & 0 & 0\\
    0 & 0 & 0 & 1 & 0 & 0 & 0 & 0 & 0 & -1 & 0 & 0 & 0 & 0 & 0\\
    0 & 0 & 0 & 0 & 1 & 0 & 0 & 0 & 0 & 1 & 0 & 0 & 0 & 0 & 0\\
    0 & 0 & 0 & 0 & 0 & 0 & 0 & 1 & 0 & 0 & 0 & 0 & 0 & 0 & 0\\
    0 & 0 & 0 & 0 & 0 & 0 & 0 & 0 & 1 & 0 & 0 & 0 & 0 & 0 & 0\\
    0 & 0 & 0 & 0 & 0 & 0 & 0 & 0 & 0 & 0 & 0 & -1 & 1 & 0 & 0\\
    0 & 0 & 0 & 0 & 0 & 2 & 0 & 0 & 0 & 0 & 0 & 0 & 0 & 0 & 0\\
    0 & 0 & 0 & 0 & 0 & 0 & 2 & 0 & 0 & 0 & 0 & 0 & 0 & 0 & 0\\
    0 & 0 & 0 & 0 & 0 & 0 & 0 & 2 & 0 & 0 & 0 & 0 & 0 & 0 & 0\\
    0 & 0 & 0 & 0 & 0 & 0 & 0 & 0 & 0 & 2 & 0 & 0 & 0 & 0 & 0\\
    0 & 0 & 0 & 0 & 0 & 0 & 0 & 0 & 0 & 0 & 2 & 0 & 0 & 0 & 0\\
    0 & 0 & 0 & 0 & 0 & 0 & 0 & 0 & 0 & 0 & 0 & 0 & 2 & 0 & 0\\
    0 & 0 & 0 & 0 & 0 & 0 & 1 & 0 & 0 & 0 & 0 & 0 & 0 & 0 & 0\\
    0 & 0 & 0 & 0 & 0 & 0 & 0 & 1 & 0 & 1 & 0 & 0 & 0 & 0 & 0\\
    0 & 0 & 0 & 0 & 0 & 0 & 0 & 0 & 1 & 0 & 1 & 0 & 0 & 0 & 0\\
    0 & 0 & 0 & 0 & 0 & 0 & 0 & 0 & 0 & 0 & 1 & 0 & 0 & 0 & 0\\
    0 & 0 & 0 & 0 & 0 & 0 & 0 & 0 & 0 & 0 & 0 & 1 & 1 & 0 & 0\\
    0 & 0 & 0 & 0 & 0 & 0 & 0 & 0 & 0 & 0 & 0 & 0 & 0 & 1 & 0\\
    0 & 0 & 0 & 0 & 0 & 0 & 0 & 0 & 0 & 2 & 0 & 0 & 0 & 0 & 0\\
    0 & 0 & 0 & 0 & 0 & 0 & 0 & 0 & 0 & 0 & 2 & 0 & 0 & 0 & 0\\
    0 & 0 & 0 & 0 & 0 & 0 & 0 & 0 & 0 & 0 & 0 & 2 & 0 & 0 & 0\\
    0 & 0 & 0 & 0 & 0 & 0 & 0 & 0 & 0 & 0 & 0 & 0 & 2 & 0 & 0\\    
    0 & 0 & 0 & 0 & 0 & 0 & 0 & 0 & 0 & 0 & 0 & 0 & 0 & 2 & 0\\
    0 & 0 & 0 & 0 & 0 & 0 & 0 & 0 & 0 & 0 & 0 & 0 & 0 & 0 & 2
    \end{array}
\right)
$$
}
Once more, the random choice of the coefficients $\alpha_{(i,j,k)}$ of $\omega$ allows us to get
$$
H_2 = 
\left(
\begin{array}{ccc}
H_{(2,5,2)}^{\omega_1} \vspace{1mm}\\
\hline
H_{(2,5,2)}^{\omega_2}\vspace{1mm}\\
\hline
H_{(2,5,2)}^{\omega_3}\vspace{1mm}\\
\hline
H_{(2,5,2)}^{\omega_4}\vspace{1mm}\\
\hline
H_{(2,5,2)}^{\omega_5}
\end{array}
\right)
$$
which, by Macaulay2, we can check to be a $(30 \times 46)-$matrix of maximal rank.
We are therefore in the hypotheses of Theorem \ref{mainthm}. This means that we can construct the matrix $M$ as follows:
being $S_0^T \cdot S_0$ an invertible $(15 \times 15)$ square matrix and, by the commutativity of the resolutions, having $S_0 \cdot M = D_{4,1} \cdot S_1$, we can compute the desired matrix as
$$
M = (S_0^T \cdot S_0)^{-1} \cdot S_0^T \cdot D_{4,1} \cdot S_1.
$$

\section{$(67 \times 125)$  matrix of linear forms and  constant rank 65  through the Horrocks Mumford bundle}

\begin{verbatim}
S = ZZ/32003[x_0..x_4];
E = ZZ/32003[e_0..e_4, SkewCommutative=>true];

alphad = map(E^5,E^{2:-2},{{e_4*e_1,e_2*e_3},{e_0*e_2,e_3*e_4},
{e_1*e_3,e_4*e_0},{e_2*e_4,e_0*e_1},{e_3*e_0,e_1*e_2}})
alpha = syz alphad
alphad' = beilinson(alphad,S)
alpha' = beilinson(alpha,S)
H = prune homology(alphad',alpha');
betti res H
regularity H
F = truncate(5,H)
betti res F
C = res F
C1 =C.dd_1;
Mnew = coker C1;
I = matrix{{x_0,x_1,x_2,x_3,x_4}}
N = coker I
P = res N^33;
A = random(S^33,S^100) 
f = inducedMap(N^33,Mnew,A)
H =res f;
rank H_1, rank H_2

W0 = syz A;
W1 = syz H_1;
W0t = transpose W0;
Winv = inverse(W0t*W0);
M = Winv*W0t*C1*W1;
\end{verbatim}

\end{document}